\documentclass[a4paper,11pt]{amsart}
\usepackage{amsmath,amsthm,amssymb}
\usepackage{comment}
\newcommand{\mychoice}[3]{#1
}
\mychoice{
\newcommand{\plabel}[1]{ \label{#1}}
\newcommand{\gbibitem}[1]{ \bibitem{#1}}
\newcommand{\snewpage}{}

\specialcomment{commentx}{}{}\excludecomment{commentx}
\specialcomment{commenty}{}{}\excludecomment{commenty}
\specialcomment{commentz}{}{}

}{
\usepackage{xcolor}
\newcommand{\plabel}[1]{ \label{#1}\rlap{\smash{${}^{^{[#1]}}$}}}
\newcommand{\gbibitem}[1]{ \bibitem{#1}\rlap{\smash{${}^{^{[#1]}}$}}}
\newcommand{\snewpage}{\newpage}

\newenvironment{commentx}{\color{magenta} }{\color{black} }
\newenvironment{commenty}{\color{blue} }{\color{black} }

}{
\usepackage{xcolor}
\newcommand{\plabel}[1]{ \label{#1}}
\newcommand{\gbibitem}[1]{ \bibitem{#1}}
\newcommand{\snewpage}{}

\specialcomment{commentx}{}{}\excludecomment{commentx}

\specialcomment{commentz}{}{}\excludecomment{commentz}

}
\DeclareMathOperator{\sgn}{sgn}

\DeclareMathOperator{\Id}{Id}
\DeclareMathOperator{\tr}{tr}

\DeclareMathOperator{\artanh}{artanh}
\DeclareMathOperator{\arcosh}{arcosh}

\DeclareMathOperator{\arsinh}{arsinh}
\DeclareMathOperator{\Rea}{Re}\DeclareMathOperator{\Ima}{Im}
\DeclareMathOperator{\CR}{CR}
\DeclareMathOperator{\DW}{DW}

\theoremstyle{definition}
\newtheorem{point}{}[section]

\newtheorem{disc}[point]{Discussion}
\newtheorem{remark}[point]{Remark}
\newtheorem{example}[point]{Example}
\theoremstyle{plain}

\newtheorem{lemma}[point]{Lemma}

\newtheorem{cor}[point]{Corollary}
\newtheorem{theorem}[point]{Theorem}
\newcommand{\bem}{\begin{bmatrix}}
\newcommand{\eem}{\end{bmatrix}}

\DeclareMathOperator{\arcoth}{arcoth}

\newcommand{\eqed}{
\pushQED{\qed}
\qedhere
\popQED
}
\newcommand{\eqedexer}{
\renewcommand{\qedsymbol}{$\diamondsuit$}
\pushQED{\qed}
\qedhere
\popQED
\renewcommand{\qedsymbol}{$\Box$}
}

\newcommand{\qedexer}{  \renewcommand{\qedsymbol}{$\diamondsuit$} \qed \renewcommand{\qedsymbol}{$\Box$}}
\newcommand{\qedremark}{  \renewcommand{\qedsymbol}{$\triangle$} \qed \renewcommand{\qedsymbol}{$\Box$}}

\newcommand{\marginextend}[1]{ \addtolength{\oddsidemargin}{-#1}  \addtolength{\evensidemargin}{-#1}
  \addtolength{\textwidth}{#1}\addtolength{\textwidth}{#1}}
\newcommand{\updownextend}[1]{ \addtolength{\topmargin}{-#1}  \addtolength{\textheight}{#1}
\addtolength{\textheight}{#1}}
\marginextend{1cm}
\updownextend{0cm}
\allowdisplaybreaks[4]
\title{The elliptical range theorem for the conformal range}
\author{Gyula Lakos}
\email{gyula.lakos@uni-miskolc.hu}
\address{Institute of Mathematics, Department of Analysis, University of Miskolc, H-3515 Miskolc-Egyetemváros, Hungary}
\keywords{ Davis--Wielandt shell, numerical range, conformal range of operators, hyperbolic geometry, elementary analytic geometry}
\subjclass[2020]{Primary: 15A60, Secondary: 51M10.}
\begin{document}
\begin{abstract}
The conformal range (or the real Davis--Wielandt shell), which is a particular planar projection of the Davis--Wielandt shell,
 can be considered as the hyperbolic version of the numerical range; i.~e. it is a ``field of values''
 which can be interpreted as a subset of the asymptotically closed hyperbolic plane.
Here we explain  the analogue of  the elliptical range theorem of $2\times2$ complex matrices for the conformal range.
If we consider the conformal range in the Beltrami--Cayley--Klein (BCK) model of the asympotically closed
 hyperbolic plane (which model just looks like as the closed unit disk on the Euclidean plane
  with the boundary points being the asymptotical points), then we find the following:
The conformal range of a $2\times2$ complex matrix is a possibly degenerate elliptical disc in the
 BCK model  but one which must avoid the particular asymptotical point $(0,1)$.
In the viewpoint of the synthetic (that is ``innerly interpreted'') geometry of the hyperbolic plane,
 these sets corresponds to the limiting objects of $h$-elliptical disks  but ``keeping away'' from the particular asymptotic point $(0,1)$.
Here `$h$' indicates that the object is interpreted from the viewpoint of hyperbolic  geometry.
More specifically, if the $2\times2$ complex matrix $A$ is  normal, then its  conformal range in the BCK  model
 is a possibly degenerate segment with its endpoints corresponding to the  eigenvalues of $A$ up to conjugation.
In the synthetic geometry of the asymptotically closed hyperbolic plane, this might be an
 $h$-point, an $h$-asymptotical point, an $h$-segment, an asymptotically closed $h$-half line, or an asymptotically closed $h$-line.
If  the $2\times2$ complex matrix $A$ is not normal, then its  conformal range in the BCK  model
 looks like an elliptical disk (allowing also circular disks).
In the synthetic geometry of the  asymptotically closed hyperbolic plane, this might be an
 $h$-elliptical disk (if $A$ has non-real eigenvalues; circular in the case of double or conjugate eigenvalues, proper otherwise),
 an $h$-elliptical parabolical disk (if $A$ has one real eigenvalue and one non-real eigenvalue),
 an $h$-distance band (if $A$ has two distinct real eigenvalues),
 or an $h$-horodisk (if $A$ has two equal real eigenvalues).
In any case, the asymptotic points correspond to real eigenvalues, and the synthetically understood
 foci correspond to eigenvalues up to conjugation.
We also supply metrical data related to the range.
We include a discussion of comparison between the elliptical range theorems
regarding the numerical range, the Davis--Wielandt shell, and the conformal range.
\end{abstract}
\maketitle
\snewpage
\section*{Introduction}
\textbf{The numerical range.}
Assume that $\mathfrak H$ is a complex Hilbert space, with interior product $\langle\cdot,\cdot\rangle$
(linear in the first variable, skew-linear in the second variable); $|\mathbf x|_2=\sqrt{\langle \mathbf x,\mathbf x\rangle}$.
If $A$ is a 
 linear operator on $\mathfrak H$, then its numerical range
(first studied by Toeplitz \cite{Toe}, Hausdorff \cite{Hau}, cf. Gustafson, Rao \cite{GR}) is defined as
\[\mathrm W(A)=
\left\{
\left(\frac{\Rea\langle \mathbf y,\mathbf x\rangle}{|\mathbf x|_2^2} ,\frac{\Ima\langle \mathbf y,\mathbf x\rangle}{|\mathbf x|_2^2} \right)
\,:\, A\mathbf x=\mathbf y, \, \mathbf x\neq 0
\right\}.\]
This is a subset of $\mathbb R^2$, usually identified with $\mathbb C$ (any point can occur there).

\textbf{The conformal range.}
A similar construction is given by the conformal range, or in other name, the real
Davis--Wieland shell
(taken from \cite{L2}, the name will be explained later).
Actually, we give two variants (with their relation explained soon):
\[\DW_{\mathrm{pCK}}^{\mathbb R}(A)=
\left\{
\left(\frac{\Rea\langle \mathbf y,\mathbf x\rangle}{|\mathbf x|_2^2}  , \frac{|\mathbf y|_2^2}{|\mathbf x|_2^2}\right)
\,:\, A\mathbf x=\mathbf y, \, \mathbf x\neq 0
\right\},\]
and
\[\DW_{\mathrm{BCK}}^{\mathbb R}(A)=
\left\{
\left(\frac{2\Rea\langle\mathbf  y,\mathbf x\rangle}{|\mathbf y|_2^2+|\mathbf x|_2^2}  ,
\frac{|\mathbf y|_2^2-|\mathbf x|_2^2}{|\mathbf y|_2^2+|\mathbf x|_2^2}\right)
\,:\, A\mathbf x=\mathbf y, \, \mathbf x\neq 0
\right\}
.\]
Here $\DW_{\mathrm{pCK}}^{\mathbb R}(A)$ will be a subset of $\overline{H}^{2-}_{\mathrm{pCK}}=\{(x,z)\in\mathbb R^2 \,:\, x^2\leq z\}$,
 and $\DW_{\mathrm{BCK}}^{\mathbb R}(A)$ will be a subset of $ \overline{H}^{2-}_{\mathrm{BCK}}=\{(x,z)\in\mathbb R^2 \,:\, x^2+z^2\leq 1\}\setminus\{(0,1)\}$
 (any point can occur in those sets).
The points $(x_{\mathrm{pCK}},z_{\mathrm{pCK}} )\in\DW_{\mathrm{pCK}}^{\mathbb R}(A)$ and
$( x_{\mathrm{BCK}}  , z_{\mathrm{BCK}})\in\DW_{\mathrm{BCK}}^{\mathbb R}(A)$
can be transcribed into each other by the canonical correspondences
\begin{equation}
\bem x_{\mathrm{BCK}}  \\ z_{\mathrm{BCK}}  \\ 1 \eem
=
\frac1{z_{\mathrm{pCK}}+1}\!
\bem 2\,x_{\mathrm{pCK}}  \\ z_{\mathrm{pCK}}-1  \\  z_{\mathrm{pCK}}+1 \eem
\,\,\,\text{and}\,\,
\bem x_{\mathrm{pCK}}   \\ z_{\mathrm{pCK}}  \\ 1 \eem
=
\frac1{1-z_{\mathrm{BCK}}}\!
\bem x_{\mathrm{BCK}}  \\ 1+z_{\mathrm{BCK}}  \\  1-z_{\mathrm{BCK}} \eem.
\plabel{eq:can20}
\end{equation}
These transcriptions, one can notice, are realized by projective transformations.

It is advantageous to augment
 $ \overline{H}^{2-}_{\mathrm{BCK}}$ by the point $(0,1)$, making
 \[\DW_{\mathrm{pCK}}^{\mathbb R}(A)\subset \overline{H}^{2 }_{\mathrm{BCK}}\equiv \{(x,z)\in\mathbb R^2 \,:\, x^2+z^2\leq 1\}.\]
According to the projective transcription \eqref{eq:can20}, this corresponds to augmenting
 $ \overline{H}^{2-}_{\mathrm{pCK}}$ by the ideal point $\infty_{\langle0,1\rangle}$
 (the ideal point in direction $\langle0,1\rangle$).
This makes \[\DW_{\mathrm{pCK}}^{\mathbb R}(A)\subset
\overline{H}^{2 }_{\mathrm{pCK}}\equiv \{(x,z)\in\mathbb R^2 \,:\, x^2\leq z\}\cup\{\infty_{\langle0,1\rangle}\}.\]
Here the extra points are never taken in the conformal range, but the augmented base sets $\overline{H}^{2 }_{\mathrm{BCK}}$
and  $\overline{H}^{2 }_{\mathrm{pCK}}$ can naturally be endowed by a geometry:
\snewpage

\textbf{ $\overline{H}^{2 }_{\mathrm{BCK}}$ and
  $\overline{H}^{2 }_{\mathrm{pCK}}$ as models of the asymptotically closed  hyperbolic plane.}
Here the ordinary $h$-points are those which are in the interior of the base set and the asymptotic
 $h$-points are those which are on the boundary of the base set.
The asymptotically closed $h$-lines are just the intersections of ordinary lines meeting the interior
 of the base set with the base set.
The ordinary $h$-lines are the asymptotically closed $h$-lines without their asymptotic $h$-points
(i.~e.~non-empty intersections of ordinary lines and the interior of the base set).
Here the natural maps are the projective $h$-collineations, which are collineations of the ambient
 projective space leaving the base set invariant   restricted of to (possibly the interior of) the base set.
(These take ordinary $h$-points to ordinary $h$-points, and asymptotic $h$-points to asymptotic $h$-points).
We will simply (and, in fact, not erroneously) call these maps $h$-collineations.
By the Beltrami--Cayley--Klein distance, an $h$-collineation invariant, natural (for connoisseurs, ``$k=1$'') $h$-metric
 is induced on the ordinary $h$-points.
This $h$-metric can be extended to the  asymptotically closed case as a $[0,+\infty]$-valued metric by
 the prescription that the distance of an asymptotic point to any other point is $+\infty$.
In that way, we obtain the   Beltrami--Cayley--Klein model of the asymptotically closed hyperbolic plane
 on   $\overline{H}^{2 }_{\mathrm{BCK}}$, and the
 parabolic Cayley--Klein model of the asymptotically closed hyperbolic plane
 on  $\overline{H}^{2 }_{\mathrm{pCK}}$.
(In what follows, the notation $\overline{H}^{2 }_{\mathrm{BCK}}$ and $\overline{H}^{2 }_{\mathrm{pCK}}$
 will be meant to include the geometry.)
The geometries of these models are isomorphic, as simple projections transformations, cf.~\eqref{eq:can20},  relate them.
Restricted to the ordinary $h$-points, we have the Beltrami--Cayley--Klein and the
 parabolic Cayley--Klein models of the hyperbolic plane, where
 the $h$-isometries are the same as $h$-collineations.
See Berger \cite{Ber} for a short introduction to hyperbolic geometry, and, in particular,
 for the Beltrami--Cayley--Klein model.
The parabolic Cayley--Klein model is a just a variant by projective equivalence.
(See Klein \cite{Kle1}, \cite{Kle2}, \cite{Kle3} for the original and detailed exposition of the projective approach.)
~
Now we can say that the  that conformal range takes values in the asymptotically closed hyperbolic plane.
What is not so clear yet is that hyperbolic geometry is indeed fundamental with respect to the behaviour of conformal range.
The fundamental properties of the conformal range can be established directly (see \cite{L2}), however
 our strategy here will be using known information regarding
\snewpage

\textbf{The Davis--Wielandt shell.}
This is a relatively classical object, cf.
Wielandt \cite{Wie0}, \cite{Wie}; Davis \cite{D1}, \cite{D2}, \cite{D3}; Li, Poon, Sze \cite{LPS}; Lins, Spitkovsky, Zhong \cite{LSZ}.
For us,  \cite{D1} and \cite{D2} are useful.
If $A$ is a linear operator on $\mathfrak H$, then its Davis--Wielandt shell is defined as
\[\DW_{\mathrm{pCK}}(A)=
\left\{
\left(\frac{\Rea\langle \mathbf y,\mathbf x\rangle}{|\mathbf x|_2^2} ,\frac{\Ima\langle \mathbf y,\mathbf x\rangle}{|\mathbf x|_2^2} , \frac{|\mathbf y|_2^2}{|\mathbf x|_2^2}\right)
\,:\, A\mathbf x=\mathbf y, \, \mathbf x\neq 0
\right\},\]
or
\[\DW_{\mathrm{BCK}}(A)=
\left\{
\left(\frac{2\Rea\langle\mathbf  y,\mathbf x\rangle}{|\mathbf y|_2^2+|\mathbf x|_2^2} ,\frac{2\Ima\langle \mathbf y,\mathbf x\rangle}{|\mathbf y|_2^2+|\mathbf x|_2^2} , \frac{|\mathbf y|_2^2-|\mathbf x|_2^2}{|\mathbf y|_2^2+|\mathbf x|_2^2}\right)
\,:\, A\mathbf x=\mathbf y, \, \mathbf x\neq 0
\right\}.\]
Here
 \[\DW_{\mathrm{pCK}}(A)
\subset
 \overline{H}^3_{\mathrm{pCK}}=\{(x,y,z)\,:\, x^2+y^2\leq z\}\cup\{\infty_{\langle0,0,1\rangle}\},\]
 and
  \[\DW_{\mathrm{BCK}}(A)
\subset
 \overline{H}^3_{\mathrm{pCK}}=\{(x,y,z)\,:\, x^2+y^2+z^2\leq1\}.\]
 $\overline{H}^{3}_{\mathrm{BCK}}$ and
  $\overline{H}^{3 }_{\mathrm{pCK}}$ can similarly be thought as models of the asymptotically closed  hyperbolic space.
Thus the sets  $\DW_{\mathrm{BCK}}(A)$ and $\DW_{\mathrm{pCK}}(A)$ are interpreted as subsets of the asymptotically closed
BCK and pCK models of the hyperbolic space.
Between the BCK and pCK models, there are (for us) canonical correspondences given by
\begin{equation}
\bem x_{\mathrm{BCK}} \\ y_{\mathrm{BCK}}  \\ z_{\mathrm{BCK}}  \\ 1 \eem
=
\frac1{z_{\mathrm{pCK}}+1}\!
\bem 2\,x_{\mathrm{pCK}} \\ 2\,y_{\mathrm{pCK}}  \\ z_{\mathrm{pCK}}-1  \\  z_{\mathrm{pCK}}+1 \eem
\,\,\,\text{and}\,\,
\bem x_{\mathrm{pCK}} \\ y_{\mathrm{pCK}}  \\ z_{\mathrm{pCK}}  \\ 1 \eem
=
\frac1{1-z_{\mathrm{BCK}}}\!
\bem x_{\mathrm{BCK}} \\ y_{\mathrm{BCK}}  \\ 1+z_{\mathrm{BCK}}  \\  1-z_{\mathrm{BCK}} \eem;
\plabel{eq:can2}
\end{equation}
 with $(0,0,1)$ in the BCK model corresponding to $\infty_{\langle0,0,1\rangle}$ in the pCK model
( although these latter points do not appear in the shell for linear operators).
\snewpage

\textbf{The relationship of the ranges.}
One can see that the Davis--Wielandt shell is a common generalization of the numerical and conformal ranges.
Or, alternatively, both the numerical and conformal ranges are particular projections of the Davis--Wielandt shell:
~
We see that $\mathrm W(A)$ can be obtained  from $\DW_{\mathrm{pCK}}(A)$ by omitting the third coordinates in the points,
 thus the numerical range it is the ``vertical'' projection of  Davis--Wielandt shell $\DW_{\mathrm{pCK}}(A)$.
If we imagine the projection as setting the third coordinate to $0$,
 then, in terms of hyperbolic geometry, this a particular projection from the never-taken asymptotic point
 $\infty_{\langle0,0,1\rangle}$  to the particular asymptotical
 tangent plane $z_{\mathrm{pCK}}=0$.
~
On the other hand, $\DW_{\mathrm{pCK}}^{\mathbb R}(A)$  can be obtained  from  $\DW_{\mathrm{pCK}}(A)$
 by omitting the second coordinates in the points, thus the corresponding conformal range is also projection of  Davis--Wielandt shell.
If we imagine the projection as setting the second coordinate to $0$,
 then, in terms of hyperbolic geometry, this a hyperbolic orthogonal projection  to the particular asymptotically closed
 hyperbolic plane  $y_{\mathrm{pCK}}=0$.
The same can be said about the relationship of $\DW_{\mathrm{BCK}}^{\mathbb R}(A)$  and  $\DW_{\mathrm{BCK}}(A)$.
~
In fact,  the numerical and conformal ranges are also related in a more direct manner as long as $A$ is bounded.
Let us define the real double (which is more like a kind of $3/2$ multiple) as
\begin{equation}
\mathrm D^{\mathbb R}(A)=\frac{A+A^*}{2}+\mathrm i A^*A.
\plabel{eq:doub}
\end{equation}
Then it is easy to see that
\begin{equation}
\DW_{\mathrm{pCK}}^{\mathbb R}(A)=\mathrm W \left(\mathrm D^{\mathbb R}(A)\right).
\plabel{eq:doubac}
\end{equation}
Consequently, the more advanced methods to compute the numerical range, like of Kippenhahn \cite{Kip1} / \cite{Kip2}
  can also be used to compute the conformal range.
(See \cite{L2} for this more directly.)
The ``easy'' relationship \eqref{eq:doubac} is, however, a bit deceiving, as
 in terms of fundamental properties of the conformal range, it is
 the Davis--Wielandt shell which will be more relevant.
(Another drawback is that \eqref{eq:doubac} cannot be well used when $A$ is an unbounded or multivalued linear operator,
 although those will not be of concern here.)
In turn, it can be thought that the Davis--Wielandt shell is behind some properties of the numerical range,
 cf.~the convexity argument of Davis \cite{Dav}; which is among the simplest ones along Dekker / Halmos \cite{Hal}.)
\snewpage

\textbf{The choice of model.}
We see that the BCK and pCK models can be used equivalently with respect to Davis--Wielandt shell and the conformal range.
Due to its simple relation to the $\mathrm W(A)$, the version $\DW_{\mathrm{pCK}} (A)$
 is the natural choice of linear algebraic community for the Davis--Wielandt shell (cf.~Horn, Johnson \cite{HJ}).
From theoretical viewpoint, in order to recognize certain symmetries, however, $\DW_{\mathrm{BCK}} (A)$ is a more convenient choice;
 furthermore, there is hardly any geometrical literature dealing with the pCK model beyond the
 general mention that, due to projective invariance, there is a certain freedom in choosing the base set.
Therefore we will use the BCK model as our main choice
.
~
One can, however, use any other model for the asymptotically closed hyperbolic space or plane,
 if we clarify what is the canonical way of transcription from the projective models.
In \cite{L2}, the conformal range was defined originally as $\CR(A)\equiv\DW^{\mathbb R}_{\mathrm{Ph}}(A)$,
 the real Davis--Wielandt shell in the Poincar\'e half-plane model.
(That model was also constructed first by Beltrami, but became popular though the work of Poincar\'e.
For the sake simplicity, we will also use the customary terminology.)
The name `conformal range' takes after its particular presentation and behaviour  in the Poincar\'e half-plane model,
 but, otherwise, it is just a synonymous expression for the `real Davis--Wielandt shell'.
~
One less familiar with hyperbolic geometry could think that using a single model would be less confusing.
This is not the case.
Various models have various strengths.
Using a single model is possible technically but that would be less enlightening.
~
At certain points,  we can make statements which are valid in more than one models,
 then we put `$*$' to the place of the model, leading to  $\DW_{*}(A)$ or $\DW^{\mathbb R}_{*}(A)$,
 signifying (at least partial) model-independence in the statement.
If a statement is meant in synthetic hyperbolic sense, then the `$*$'s can be omitted altogether.
A possibly asymptotic $h$-point $\boldsymbol x$ is understood as a compatible collection of $\boldsymbol x_*$'s;
 $\boldsymbol x_{\mathrm{pCK}}$ is the tuple of its coordinates in the pCK model, etc.

\textbf{The setting of the elliptical range theorems.}
The simplest yet not entirely trivial case for the numerical and conformal ranges is when $A$ is a complex $2\times2$ matrix.
In that case, the numerical range is described geometrically by the well-known elliptical range theorem.
This says that $\mathrm W(A)$ is a possibly degenerate elliptical disk with foci given by the eigenvalues
 of $A$, with major and minor semi-axes $s^\pm=\sqrt{\dfrac{U_A\pm|D_A|}2}$,
 and half focal distance $s^{\mathrm f}=\sqrt{|D_A|}$ (where the values
 $U_A$ and $|D_A|$ are defined in Subsection \ref{ssub:spectype}).
One can see Toeplitz \cite{Toe} (1918), Murnaghan \cite{Mur}, Donoghue \cite{Don},
 Horn, Johnson \cite{HJ}, Gustafson, Rao \cite{GR}, Li \cite{Li}, for various proofs (up to various detail).
We are partial, however, for \cite{LL}.
See Johnson \cite{Joh} (real case) and Uhlig \cite{Uhl} for the presentation of the analytic geometric equation
 (when $s^->0$, which happens when $A$ is non-normal).
Note, however, that the main point in the elliptical range theorem is the geometrical interpretation.
~
One may wonder whether a similar elliptical range theorem is valid with respect to the conformal range.
Indeed, the answer is affirmative:

\snewpage
\textbf{The qualitative elliptical range theorems in general.}
It is easy to prove that for a complex $2\times2$ matrix $A$, the
Davis--Wielandt shell $\DW_{\mathrm{pCK}}(A)$ or $\DW_{\mathrm{BCK}}(A)$ is a possibly degenerate ellipsoid
(See Davis \cite{D2}, Theorem 10.1.)
This is a kind of elliptical range theorem already.
This immediately implies that the numerical range $\mathrm W(A)$ is a possibly degenerate elliptical disk.
Moreover, it also implies that the conformal ranges $\DW_{\mathrm{pCK}}^{\mathbb R}(A)$ and
 $\DW_{\mathrm{BCK}}^{\mathbb R}(A)$ yield possibly degenerate elliptical disks
(in Euclidean view)
.
Alternatively, the ordinary elliptical range theorem and \eqref{eq:doubac} implies that $\DW_{\mathrm{pCK}}^{\mathbb R}(A)$
 is a possibly degenerate disk; and by projective transition, so is $\DW_{\mathrm{BCK}}^{\mathbb R}(A)$.
In this most qualitative level, the elliptical range theorems are quite trivial.
Therefore, any interesting statement should involve more precise (in our case: metric and focal) information.
~
One can prove that $A$ is normal if and only if $\mathrm D^{\mathbb R}(A)$ is normal.
Thus  $\DW_{\mathrm{pCK}}^{\mathbb R}(A)$ and $\DW_{\mathrm{pCK}}^{\mathbb R}(A)$
 are degenerate (meaning a point or a segment) if and only if $A$ is normal.
Regarding the focal properties, however, there are no obvious consequences concerning the conformal range.
In the non-normal case, $A\mapsto \mathrm D^{\mathbb R}(A)$ does not even respect parabolicity in any direction.
The point, however, is that the conformal range  and the Davis--Wielandt shell should be interpreted in terms of hyperbolic geometry.
~
The Davis--Wielandt shell is quite different from the numerical range, but the elliptical range theorem for it
 is quite relatable to the case of the numerical range.
~
The elliptical range theorem for the conformal range is slightly more complicated than for the numerical range.
This can be attributed 
to the fact that asymptotic points can act as focal points in our hyperbolic conics.
Otherwise, the analogy is close.
The only unusual feature is that we have to consider conics on the hyperbolic plane, which is a slightly obscure topic.
Nevertheless, the study of hyperbolic conics predates the study of the numerical range:

\textbf{Hyperbolic conics in general.}
The systematic study of hyperbolic conics was apparently started by Story \cite{Sto} (1882), a rather serious work.
Then hyperbolic conics are discussed by Killing \cite{Kil} (1885)
 D'Ovidio \cite{DO1}, \cite{DO2}, \cite{DO3} (1891), Barbarin \cite{Bar} (1901), Liebmann \cite{Lib0} (1902),
 Massau \cite{Mas} (1905), Liebmann \cite{Lib1} (1905), Coolidge \cite{Coo} (1909).
A very detailed study of the conics is conducted by V\"or\"os \cite{V1} / \cite{V2} (1909/10).
Thus, it can be said that by 1910 the hyperbolic conics were rather well-explored;
 although not all of the sources cited above can be considered equally accessible.
(A parallel record for quadratic surfaces is provided by
 Barbarin \cite {Bar}, Coolidge \cite{CooP} (1903), Bromwich \cite{Bro} (1905), Coolidge \cite{Coo},
 and V\"or\"os \cite{V3} (1912).)
Numerical aspects of the classification of hyperbolic conics are discussed further in Fladt \cite{F1}, \cite{F2}, \cite{F3}.
For those who are not specialists of hyperbolic geometry, the exposition of Izmestiev \cite{Izm} may be a good choice for a start.
(For us, \cite{Izm} Section 6.4 will be important.
But, for one truly interested in hyperbolic conics, facing Story \cite{Sto} is probably unavoidable.)
Ultimately, the point is that $h$-conics are simplest to be defined as conics of the projective models (i.~e.~of pCK or BCK).
Then they enjoy similar synthetic presentations (essentially: focal descriptions) and properties as in the Euclidean case.
Nevertheless, requiring any substantial preliminary knowledge about hyperbolic conics would likely be too optimistic.
Therefore, despite having given out several references, we organize the presentation so that
 that familiarity with hyperbolic conics  beyond cycles  will be only helpful but not required.
\snewpage

\textbf{Preliminaries required.}
For our purposes, a basic understanding of hyperbolic geometry up to cycles,
 e.~g.~as in Berger \cite{Ber}, would be sufficient.
(But there are several very nice model based treatments.)
There reader should also be able to recognize some hyperbolic symmetries in the BCK model.
(This amounts more or less just to the observation that the Euklidean symmetries of the BCK model are also $h$-collineations.)
Apart from that, we quote only three basic statements about the Davis--Wielandt range from Davis \cite{D1}, \cite{D2};
 therefore no particular familiarity with the Davis--Wielandt range will be required either.
\snewpage

\textbf{A reminder on cycles.}
Cycles on the hyperbolic plane can already be defined using $h$-collineations:
If $P$ is an ordinary $h$-point, then the orbit of an external ordinary $h$-point $E$ under the $P$-fixing
 $h$-collineations is the $h$-circle with center $P$ passing through $E$.
If $P_0$ is an asymptotic $h$-point, then the orbit of an ordinary $h$-point $E$ under the
the group generated by those $h$-collineations whose fixed point set is
an asymptotically closed $h$-line passing through $P_0$
 is the $h$-horocycle with asymptotic point $P_0$ passing through $E$.
If $l$ is an $h$-line, then orbit of an external ordinary $h$-point $E$ under the $h$-collineations  fixing the
 asymptotic points of $l$ is the (two-sided) $h$-distance line with axis $l$ passing through $E$.
In terms of the associated distance geometry:
An $h$-circle is the set of points of a given positive distance from $P$.
An $h$-horocycle is a set of points of infinite but ``equal'' distance from $P_0$.
(The points $E_1\neq E_2$ are ``of equal distance from $P_0$'' if  the asymptotic triangle $E_1E_2P_0\bigtriangleup$
 is $h$-congruent to $E_2E_1P_0\bigtriangleup$; or equivalently, the asymptotic closure
 of the perpendicular bisector of the segment $E_1E_2$ passes through $P_0$; this leads to an equivalence relation.
Regularized  distances from $P_0$ can also be defined;
 but these are more of less equivalent to $h$-horocycles with asymptotical point $P_0$, as they are just signed distances from them.)
The $h$-distance line is the set of points of a given positive distance from $l$.
In general, there is no harm in taking the distance-based definitions for $h$-circles and $h$-distance lines;
 however this is not done directly in the case of $h$-horocycles,
 so we take the definition with $h$-collineations here.
The corresponding $h$-disks, $h$-horodisks, $h$-distance bands are obtained by taking convex closure,
 or, alternatively, allowing distances less or equal than the original threshold level;
 that is, essentially, by filling the interiors in.
The asymptotically closed versions are obtained by adding $P_0$ or the asymptotic points of $l$, depending on the case.

One can describe $h$-cycles analytically in BCK and pCK models as follows:
In the ambient projective spaces of these models, the asymptotically closed cycles are given
 by non-degenerate quadrics which have at least one ordinary $h$-point
  and whose equations are linearly generated by the quadric of the asymptotic points (called: absolute)
  and by the quadric of a double line.
In the case of the $h$-circle, $h$-horocycle, $h$-distance line, the double line is supported on a
 projective line intersecting the set of asymptotic points in 0,1,2 points, respectively.
In the case of the $h$-distance line, the line is the projective extension of the defining line $l$.
In the case of the $h$-horocycle, the line is tangent to the asymptotic point $P_0$.
In the case of the $h$-circle, the line is external to the model (it is the polar line of $P$ relative to the absolute).
This characterization can easily be proven considering some nice canonical representatives, and then
 projective invariance extends it generally.
Although we make remarks in this regard, we do not use the this analytical characterization
 in merit, as we will rather consider the canonical representatives directly.
Nevertheless, from a purely analytic viewpoint, the analytic characterization above is the simplest and most unified one
 to introduce $h$-cycles.
We also remark:
The $h$-distance lines are easy to recognize
 because they are tangent ellipses at the two critical asymptotic points to the set of asymptotic points from the inside.
The $h$-horocycles are not only osculating but hyperosculating ellipses to the set of asymptotic points from the inside
(inner plus osculating implies hyperosculating);
 this is a good reminder regarding their shape (unless the asymptotic point is $\infty_{\langle0,1\rangle}$,
 because osculation there is transparent only after a suitable change of coordinates).
\begin{commentx}
 The short characterizations given above, however, do not substitute the proper understanding of $h$-cycles,
 for what the reader is better to consult the standard literature if needed.
\end{commentx}

\textbf{Hyperbolic conics for our interest.}
Beyond lineal objects and cycles, we will consider only proper $h$-ellipses and $h$-elliptic parabolas.
These are most easily accounted in the analytic geometric viewpoint:
Now, indeed, if $\DW_{\mathrm{BCK}}^{\mathbb R}(A)$ is not a possibly degenerate segment, then it must be an elliptical
 subdisk  strictly contained in the unit disk (as $(0,1)_{\mathrm{BCK}}$ is forbidden for the conformal range).
As such it may have $0$, $1$, or $2$ asymptotic (i.~e.~boundary) points.
In the case of $0$ asymptotic points, we have the case of an $h$-ellipse, which may be an $h$-circle (which is a cycle)
 or it is otherwise called a proper $h$-ellipse.
In the case of $1$ asymptotic points, it may or may not be hyperosculating  to the boundary.
In the hyperosculating case we have an $h$-horocycle, and otherwise it called an $h$-elliptic parabola.
In the case of $2$ asymptotic points,  we have a pair of (conjugate) hypercyles (or ``an algebraically closed hypercycle'').
(More precisely, we have the disks associated to these conics.)
The \textit{analytic} descriptions above will be our official definitions to $h$-ellipses and $h$-elliptic parabolas.

 \textbf{Synthetical characterizations.}
The $h$-conics above which are not cycles can also be described in synthetical ways,
 by focal characterizations, representing what should be known
 about them having a synthetical understanding.
(We will not assume that knowledge but we will provide some alternative proofs for those who already have it.)

For $h$-ellipses, the synthetical characterization in terms of distance sums is already articulated by Story \cite{Sto}.
Combined with elementary observations it can be put into the following form
(but see Killing \cite{Kil}, Liebmann \cite{Lib0}, \cite{Lib1} for more synthetic discussions):
Let $F_1$ and $F_2$ be (non-asymptotic) hyperbolic points.
Let $\mathrm d$ denote the (standard) hyperbolic distance, and we choose $m^+>\mathrm d(F_1,F_2)$ the intended
``length of major axis''.
Then the associated $h$-ellipse has points $\boldsymbol x$, where $\mathrm d(F_1,\boldsymbol x)+\mathrm d(F_2,\boldsymbol x)=m^+$.
The corresponding $h$-elliptical disk is given by the equation   $\mathrm d(F_1,\boldsymbol x)+\mathrm d(F_2,\boldsymbol x)\leq m^+$.
Here $F_1$ and  $F_2$ are the foci (just like $m^+$, they are determined by the $h$-ellipse).
What makes the $h$-ellipse proper is just $F_1\neq F_2$, otherwise we have a cycle, namely an $h$-circle.
Now the statement is that analytic and synthetic $h$-ellipses are the same.
The reader may or may not know this already.
In any case, we will also demonstrate the equivalence of definitions.
However, assuming having familiarity with the  synthetic  geometry simplifies certain proofs.
Regarding $h$-ellipses, a notable geometric fact is the following:
Assume that $s^+=\frac12m^+$ is half of the length of the major axis,
Assume that $s^-=\frac12m^-$ is half of the length of the minor axis,
and $s^{\mathrm f}=\frac12m^{\mathrm f}$ is half of the distance of the foci.
Then
\begin{equation}
\cosh s^+=(\cosh s^-)(\cosh s^{\mathrm f})
\plabel{eq:Pellip}
\end{equation}
holds.
(This is just the hyperbolic Pythagorean rule with respect to the center, one focus, and one vertex on the minor axis.)

For $h$-elliptical parabolas the situation is not generic, it is better to use the viewpoint
 put forward by Killing \cite{Kil}, and discussed in greater detail by Liebmann \cite{Lib0}, \cite{Lib1}:
Let $C$ be a horocycle.
For a hyperbolic point $\boldsymbol x$, let  $\mathrm d_C(\boldsymbol x)$ denote the signed distance
 of $\boldsymbol x$ from $C$ such that it is negative if $\boldsymbol x$ is inside the horocycle $C$.
(Therefore the function $\boldsymbol x\mapsto \mathrm d_C(\boldsymbol x)$ is a regularized distance function
 from the asymptotical point of $C$. Adding a constant to this function means just passing to a parallel
 horocycle $C'$.)
Having a horocycle $C$, we can choose an arbitrary point $F$ inside the horocyle (i.~e.~$\mathrm d_C(F)<0$).
Then the (ordinary) points of the associated $h$-elliptic parabola are those hyperbolic points $\boldsymbol x$
 for which the equation $\mathrm d(F,\boldsymbol x)+\mathrm d_C(\boldsymbol x)=0$ holds.
The corresponding $h$-elliptical parabolical disk is given by the equation $\mathrm d(F,\boldsymbol x)+\mathrm d_C(\boldsymbol x)\leq0$.
Here $F$ is a focus of the $h$-elliptic parabola; but the asymptotic point $F_0$ of the horocyle
 can be considered as the other.
The tip (vertex) $V$ of the $h$-elliptic parabola lies halfway between the focus $F$
and the pedal point $P$ of $ F$ on the horocycle $C$.
As the absolute distance of the horocycle $C$ and the focus $F$ is $-\mathrm d_C(F)$,
 his implies that the distance of $V$ and $F$ is $\mathrm d(V,F)=-\frac12\mathrm d_C(F)$.
Note that, similarly to the case $h$-ellipses, beyond the foci, there is an additional metric
 data which manifests in the concrete choice of the horocycle, but
 focus-vertex distance $-\frac12\mathrm d_C( F)>0$ can equivalently be used for this purpose.
As before, one may or may not be familiar with the equivalence
 of the analytical and synthetical definitions of $h$-elliptic parabolas,
 but knowing it simplifies some proofs.

As $h$-elliptic parabolas are less familiar objects than $h$-ellipses, we will elaborate a bit more on them:
We address that how they can considered as limits of $h$-ellipses canonically.
Let us consider the $h$-elliptical parabolical disk given  as before.
Then, for $u>0$, let $ F_1(u)= F$; let $ F_2(u)$ be the point which lays
 on the half-line $\overrightarrow{FF_0}$ in hyperbolic distance $u$ from $F$.
Let $m^+(u)=-\mathrm d_C(F)+u$.
As $u\nearrow+\infty$, we see that $F_1(u)=F$, and $F_2(u)\rightarrow F_0$,
 and the functions
\[\boldsymbol x\mapsto \mathrm d(F_1(u),\boldsymbol x) + \mathrm d(F_2(u),\boldsymbol x)-m^+(u)
\equiv \mathrm d(F,\boldsymbol x)+\mathrm d(F_2(u),\boldsymbol x)-u+\mathrm d_C(F)\]
decreases (non-strictly) in $u$ (which is a simple consequence of the triangle inequality), limiting to the function
\[\boldsymbol x\mapsto  \mathrm d(F,\boldsymbol x)+\mathrm d_C(\boldsymbol x).\]
This means that, as $u\nearrow+\infty$, we have larger and larger $h$-elliptical discs, which the same major axis line;
 and due to the non-vanishing
 gradient of the function $\mathrm d_C$, they limit to our  $h$-elliptical parabolical disk in question.
This we call as the canonical limiting process.
This has he advantage that the tip (vertex on the major axis) closest to $F_1(u)=F$ is $V_1(u)=V$ stationarily.
It might be surprising, but as $u\nearrow+\infty$, the minor axis $m^-(u)$ of the $h$-ellipses increases
 but limits to a finite number, $m^-(u)\nearrow M^-$.
Indeed, for $s^-(u)=\frac12m^-(u)$, $s^+(u)=\frac12m^+(u)=\mathrm d(V,F)+\frac u2$, $s^{\mathrm f}(u)=\frac12m^{\mathrm f}(u)=\frac u2$,
 the identity $\cosh s^-(u)=\frac{\cosh s^+(u)}{\cosh s^{\mathrm f}(u)}$ limits, as $u\nearrow+\infty$, to
 $\cosh s_{\mathrm{a}}^-=\exp \mathrm d(V,F)$.
This $s_{\mathrm{a}}^-$ is the ``asymptotic'' minor semi-axis length, $s_{\mathrm{a}}^-=\frac12m_{\mathrm{a}}^-$.
As the major axes of the $h$-ellipses remain  the same, namely the axis of the $h$-elliptic parabola, we see
 the following: the $h$-elliptic parabola is contained in the distance band of radius
 $s_{\mathrm{a}}^-=\frac12m_{\mathrm{a}}^-$ around the axis but not in a smaller one.
For an $h$-elliptic parabola given on its own,
 $s^-$ and $m^-$ can be used instead of $s_{\mathrm{a}}^-$ and $m_{\mathrm{a}}^-$, respectively;
 but these are just asymptotic lengths and no actual axis segment or axis line belongs to them.
Ultimately, our main take is
\begin{equation}
 \cosh s^-=\exp \mathrm d(V,F).
\plabel{eq:Pellipar}
\end{equation}
\snewpage

\textbf{Quadratic surfaces in 3-dimensional space for our purposes.}
In terms of quadratic surfaces, for the Davis--Wielandt shell, only $h$-tubes and $h$-horospheres  will be considered
in the hyperbolic space.
~
If $P_0$ is an asymptotic $h$-point, then [the orbit of an ordinary $h$-point $E$ under the
 group generated by those $h$-collineations whose fixed point set is
 an asymptotically closed $h$-plane passing through $P_0$]
 is the $h$-horosphere with asymptotic point $P_0$ passing through $E$.
If $l$ is an $h$-line, then [the orbit of an external ordinary $h$-point $E$ under the $h$-collineations fixing the
 asymptotic points of $l$] is the $h$-tube with axis $l$ passing through $E$.
The $h$-horosphere is a set of points of infinite but ``equal'' distance from $P_0$.
The $h$-tube is the set of points of a given positive distance from $l$.
Again, taking the distance-based definition for $h$-tubes is reasonable, but the description of horospheres
 is better to be discussed in the setting of the $h$-collineation picture, as proper regularized
 distances from an asymptotic point sort of invoke horospheres already.
The asymptotically closed versions of the $h$-horospheres and  $h$-tubes
 are obtained by adding the asymptotic points $P_0$, or the asymptotic points of the lines $l$, respectively.
~
One can describe $h$-horospheres and $h$-tubes analytically in BCK and pCK models as follows:
In the ambient projective spaces of these models, the asymptotically closed $h$-horosphere is given by a
 non-degenerate quadric which have at least one ordinary $h$-point
 and whose equation is linearly generated by the quadric of the asymptotic points
 and by the quadric of a double plane tangent to the set of asymptotic points (at the asymptotic point $P_0$ of the $h$-horosphere).
The asymptotically closed $h$-tube is given by a
  non-degenerate quadric which have at least one ordinary $h$-point
  and whose equation is linearly generated by the quadric of the asymptotic points
  and by the quadric corresponding to
  a pair of different planes
  tangent to the set of asymptotic points (at the asymptotic points of the axis $l$ of  the $h$-tube).
These characterizations can be established using some canonical representatives and
 then they can be extended by projective invariance.
Again, although we make remarks on them,
 we will not really use these analytic characterizations, as we will deal with the
 canonical representatives directly.
\snewpage

\textbf{On the layout of this paper.}
Section \ref{sec:DW} reviews  basic information related to  Davis--Wielandt shell in greater detail,
 and states the elliptical range theorem there.
Section \ref{sec:CR} discusses basic information related to the conformal range,
 and states the elliptical range theorem there.
Section \ref{sec:spec22} discusses a spectral classification of the complex $2\times2$ matrices.
Section \ref{sec:DWP} proves the    elliptical range theorem for the Davis--Wielandt shell.
Section \ref{sec:CRP} proves the   elliptical range theorem for the conformal range.
Section \ref{sec:com} compares the elliptical range theorems
 for the numerical range, the Davis--Wielandt shell, and the conformal range.

\textbf{On terminology and notation.}
Due to their nature, depending on their appearance,
 the terms `$h$-collineations', `$h$-isometries', `$h$-congruences', `conformal transformations',
 are used quite ecclectically.
Also, the expressions
 `M\"obius transformations', `fractional linear transformations'
 (working synonymously in the 2-dimensional case but serving only the orientation-preserving case in the 3-dimensional case)
 belong here.
If $a>0$, then $\frac a0=+\infty$.
However, $\frac 00=\pm\infty$, for which, in formulas,
 special evaluation rules apply, which will be indicated; typically, the values $0$, $1$, or $\infty$ may apply.
The triple ratio $a:b:c$ is the 3-tuple $(a,b,c)$ up to multiplication by non-zero scalars.
(Thus $0:0:0$ is a valid ratio, and it different from any other one.)
Similar comment applies for ordinary, i.~e.~double ratios.
The inequality $A\leq B$ is fully equivalent to $C\leq D$
 if  $A< B$ is equivalent to $C< D$ and $A= B$ is equivalent to $C= D$.
$A\leq B$ transcribes to $C\leq D$ if
$A$ can be rewritten as $C$, and $B$ can be rewritten as $D$.
(This implies full equivalence.)

\textbf{On this presentation.}
The reader will recognize that much of the paper deals with the related hyperbolic geometry explained;
 that could have been avoided at the cost of making the presentation less accessible.
Also, allowing $A$ to be a multivalued operator (as in Davis \cite{D1}, \cite{D2})
 would have simplified some statements, and would have allowed to access a different
 set canonical representatives, but it would have made the treatment a bit alien
 compared to   standard linear algebra.
The emphasis is on making the material, in particular, some hyperbolic geometry, available for linear algebraists.
Although the conformal range is more informative in the higher dimensional case,
 this might be a good opportunity to develop some initial familiarity with it.

\textbf{Acknowledgements.}
The author thanks \'Akos G. Horv\'ath for help related to hyperbolic conics.
The present state of the manuscript is due to the kind advices
 of Froilán M. Dopico, and some other persons I am not able to name.
\snewpage

\snewpage
\section{On the Davis--Wielandt shell}
\plabel{sec:DW}

\subsection{On the hyperbolic space (review)}\plabel{ss:hypspace}
~\\

\textbf{Complex parametrization and collineations.}
As it was indicated, a modern review of hyperbolic geometry is contained in  Berger \cite{Ber}.
But, for our purposes, the discussions in Davis \cite{D1}, \cite{D2} are very useful
(or see \cite{L2} Appendix A).
Let us recall that the complex numbers are embedded to the pCK and BCK models by
\[\iota_{\mathrm{pCK}}(\lambda)=\left(\Rea\lambda,\Ima\lambda,|\lambda|^2\right)
\quad\,\,\text{and}\quad\,\,
\iota_{\mathrm{BCK}}(\lambda)=\left(\frac{2\Rea\lambda}{|\lambda|^2+1},
\frac{2\Ima\lambda}{|\lambda|^2+1},\frac{|\lambda|^2-1}{|\lambda|^2+1}\right)\]
 (as asymptotic points);
 and by extending to the Riemann sphere, $\iota_{\mathrm{pCK}}(\infty)=\infty_{\langle0,0,1\rangle}$, $\iota_{\mathrm{BCK}}(\infty)=(0,0,1)$,
 all asymptotic points are captured.

In what follows, in this section, `$*$' may mean either BCK or pCK.

If $f:\lambda\mapsto \frac{a\lambda+b}{c\lambda+d}$, $a,b,c,d\in\mathbb C$, $ad-bc\neq0$ is a complex fractional linear
(i.~e.~complex M\"obius) transformation, then
there is a single unique collineation $f_*$ of the asymptotically closed hyperbolic space such that
$f_*(\iota_*(\lambda))=\iota_*(f(\lambda))$ for $\lambda\in\mathbb C\cup\{\infty\}$.
This, by the collineation structure, induces, uniquely, an orientation-preserving collineation of the hyperbolic space.
(If we want to consider not orientation-preserving $h$-collineations,
we should include the ones induced from the actions $f:\lambda\mapsto \frac{a\bar\lambda+b}{c\bar\lambda+d}$, $a,b,c,d\in\mathbb C$, $ad-bc\neq0$.
Note that the effect from the transformation $\lambda\mapsto \bar\lambda$
is $(x_*,y_*,z_*)\mapsto (x_*,-y_*,z_*)$, which is the $h$-reflection to the canonically embedded $h$-plane.)
Recall that a $h$-collineation is a restriction of a projective transformation to the base set of the pCK and BCK models.
(The projective transformations, i.~e.~essentially the explicit actions on the non-asymptotical points, are written out in, say, \cite{L2} Appendix A; but we will not use them explicitly; the main point is their existence.)
Here orientation-preservance in hyperbolic sense is the same as the one in Euclidean sense
in the interior of the model, or  just on the asymptotic points.
\snewpage

\textbf{Distance.}
In the BCK model,  the natural distance function  is given by
\begin{multline}\mathrm d^{\mathrm{BCK}}((x_{\mathrm{BCK}},y_{\mathrm{BCK}},  z_{\mathrm{BCK}}),
 (\tilde x_{\mathrm{BCK}}, (\tilde y_{\mathrm{BCK}},\tilde z_{\mathrm{BCK}}))=\\
=\arcosh\left(\frac{1-x_{\mathrm{BCK}}\tilde x_{\mathrm{BCK}}-y_{\mathrm{BCK}}\tilde y_{\mathrm{BCK}} - z_{\mathrm{BCK}}\tilde z_{\mathrm{BCK}}  }{\sqrt{1-(x_{\mathrm{BCK}})^2-(y_{\mathrm{BCK}})^2-(z_{\mathrm{BCK}})^2}
 \sqrt{1-(\tilde x_{\mathrm{BCK}})^2-(\tilde y_{\mathrm{BCK}})^2-(\tilde z_{\mathrm{BCK}})^2}}\right).
\plabel{eq:dis3}
 \end{multline}
~

\textbf{Representatives for horospheres.}
Assume that the map
$F:\mathbf w\in\mathbb R^2\mapsto \mathbf A\mathbf w+\mathbf b$
 is an Euclidean isometry.
Then
$\mathrm{Up}_F: (\mathbf w,z)\in\mathbb R^2\times\mathbb R\mapsto (\mathbf A\mathbf w+\mathbf b,z+2
\langle\mathbf A\mathbf w,\mathbf b\rangle+|\mathbf b|^2)$
 induces a $h$-collineation of the pCK model (corresponding to a
 $h$-collineation   $f_{\mathrm{pCK}}$ where $f(\lambda)=a\lambda+b$ or $f(\lambda)=a\bar\lambda+b$ with $|a|=1$),
 fixing $\infty_{\langle0,0,1\rangle}$.
In fact, the  $\mathrm{Up}_F$'s form the group $F_{\infty,\infty}$ generated by those $h$-collineations
 whose fixed point set is an asymptotically closed $h$-plane through $\infty_{\langle0,0,1\rangle}$
 (corresponding to those $F$'s which are Euclidean reflections).
(It is hard to express that $\infty$ is a ``double fixpoint''.)
The group $F_{\infty,\infty}$ acts transitively on the sets with equation
\begin{equation}
z_{\mathrm{pCK}}-(x_{\mathrm{pCK}})^2-(y_{\mathrm{pCK}})^2=c;
\plabel{eq:horoc}
\end{equation}
thus one can see that these sets with $c>0$ are the $h$-horospheres with asymptotic point $\infty_{\langle0,0,1\rangle}$.
In fact, $\mathrm{Up}_F$ just lifts the action of $F$ up to all sets with equation \eqref{eq:horoc}.
(Note that the group of all $\infty_{\langle0,0,1\rangle}$-fixing collineations, which is larger,
 is the semidirect product of $F_{\infty,\infty}$  with the group of maps
 $(\mathbf w,z)\in\mathbb R^2\times\mathbb R\mapsto (a\mathbf w,a^2z)$, for $a>0$,
 corresponding to $h$-collineations  $f_{\mathrm{pCK}}$ where $f(\lambda)=a\lambda$;
 these will mix up the $h$-horospheres in question.)
(A geometer would make this argument using the closely related Poincaré half-space model, which we have not considered.)
Transcription of \eqref{eq:horoc} to the BCK model yields
\begin{equation}
(x_{\mathrm{BCK}})^2+(y_{\mathrm{BCK}})^2+ (z_{\mathrm{BCK}})^2-1 +c(z_{\mathrm{BCK}}-1)^2=0 ,
\plabel{eq:horocc}
\end{equation}
with $c>0$, as the equations for asymptotically closed $h$-horospheres with asymptotic point $(0,0,1)$.
(For the analytically minded, we remark that \eqref{eq:horocc} is indeed, a linear, in fact, positive, combination, of the equations
 $(x_{\mathrm{BCK}})^2+(y_{\mathrm{BCK}})^2+(z_{\mathrm{BCK}})^2-1=0$
 and $(z_{\mathrm{BCK}}-1)^2=0$, which are for the absolute and the double tangent plane at $(0,0,1)$ to the BCK model.
In the pCK model, as we have $ z_{\mathrm{pCK}}-(x_{\mathrm{pCK}})^2-(y_{\mathrm{pCK}})^2=0$
 for the absolute, and $0=1$ for the double ideal plane.)
In particular, $c=1$ yields
\begin{equation}
(x_{\mathrm{BCK}})^2+(y_{\mathrm{BCK}})^2+2(z_{\mathrm{BCK}})^2-2z_{\mathrm{BCK}}=0
\plabel{eq:horo}
\end{equation}
as the equation for an asymptotically closed $h$-horosphere with asymptotic point $(0,0,1)$.
\snewpage

\textbf{Representatives for $h$-tubes.}
Assume that $\mathbf A$ is $2\times2$ real matrix which is a similarity, i.~e.~positive scalar times an orthogonal transformation.
Then  the map
$\mathrm{Up}_{\mathbf A}: (\mathbf w,z)\in\mathbb R^2\times\mathbb R\mapsto (\mathbf A\mathbf w ,\|\mathbf A\|z )$
 induces a $h$-collineation of the pCK model (corresponding to a
 $h$-collineation   $f_{\mathrm{pCK}}$ where $f(\lambda)=a\lambda $ or $f(\lambda)=a\bar\lambda $),
 fixing the asymptotic points $(0,0,0)$ and $\infty_{\langle0,0,1\rangle}$.
In fact, the  $\mathrm{Up}_{\mathbf A}$'s form the group $F_{0\infty}$
containing  those $h$-collineations which fix the asymptotic points $(0,0,0)$ and $\infty_{\langle0,0,1\rangle}$ of the line
 $x_{\mathrm{pCK}}=y_{\mathrm{pCK}}=0$.
The group $ F_{0\infty}$ acts transitively on all sets with equation
\begin{equation}
z_{\mathrm{pCK}}=(1+c)(x_{\mathrm{pCK}})^2+(y_{\mathrm{pCK}})^2);
\plabel{eq:doroc}
\end{equation}
but the points $(0,0,0)$ and $\infty_{\langle0,0,1\rangle}$ taken out.
Thus one can see that these sets with $c>0$ are the $h$-tubes belonging to the $h$-line $x_{\mathrm{pCK}}=y_{\mathrm{pCK}}=0$.
(They are asymptotically closed if $(0,0,0)$ and $\infty_{\langle0,0,1\rangle}$ are added.)
Transcription of \eqref{eq:doroc} to the BCK model yields
\begin{equation}
(1+c)(x_{\mathrm{BCK}})^2+(1+c)(y_{\mathrm{BCK}})^2+ (z_{\mathrm{BCK}})^2-1 =0 ,
\plabel{eq:dorocc}
\end{equation}
with $c>0$, as the equations for asymptotically closed
 $h$-tubes belonging to the $h$-line $x_{\mathrm{BCK }}=y_{\mathrm{BCK }}=0$, i.~e.~with asymptotical points $(0,0,\pm1)$.
(For the analytically minded, we remark that \eqref{eq:dorocc} is indeed a particular linear combination, of the equations
 $(x_{\mathrm{BCK}})^2+(y_{\mathrm{BCK}})^2+(z_{\mathrm{BCK}})^2-1=0$
 and $(z_{\mathrm{BCK}})^2-1=0$, which are for the absolute and
 for the union of tangent planes   at $(0,0,\pm1)$ to the BCK model.
In the pCK model,  we have $ z_{\mathrm{pCK}}-(x_{\mathrm{pCK}})^2-(y_{\mathrm{pCK}})^2=0$
 for the absolute, and $0=z_{\mathrm{pCK}}$ for the pair of the  plane tangent to the asymptotical points at $(0,0,0)$ and the ideal plane.)

The purely distance approach works out as follows:
By symmetry reasons, the $h$-plane $z_{\mathrm{BCK}}=C$ must be perpendicular to the $h$-line $x_{\mathrm{BCK }}=y_{\mathrm{BCK }}=0$.
Therefore in the BCK model the map
$(x_{\mathrm{BCK}},y_{\mathrm{BCK}},z_{\mathrm{BCK}} )\mapsto (0,0,z_{\mathrm{BCK}} )$
must be the orthogonal (pedal) projection to the $h$-line $x_{\mathrm{BCK }}=y_{\mathrm{BCK }}=0$.
Therefore the equation for the respective $h$-tube with radius $r>0$ is
\[\mathrm d^{\mathrm{BCK }}\left((x_{\mathrm{BCK}},y_{\mathrm{BCK}},z_{\mathrm{BCK}} ), (0,0,z_{\mathrm{BCK}} )\right)=r;\]
that is
\[\frac{\sqrt{1-(z_{\mathrm{BCK}})^2}}{\sqrt{1-(x_{\mathrm{BCK}})^2-(y_{\mathrm{BCK}})^2-(z_{\mathrm{BCK}})^2}}=\cosh r.\]
Rewriting this (on ordinary $h$-points) leads to the equation \eqref{eq:dorocc} with $c=\frac1{(\sinh r)^2}$.

\snewpage
\subsection{The Davis--Wielandt shell (review)}
\plabel{ss:DWbas}~\\

Let $A$ be a linear operator on a complex Hilbert space, of any dimension.
\begin{theorem}[Wielandt, Davis]
\plabel{thm:DWronc}
If $U$ is unitary, then  $\DW_{*}(A)=\DW_{*}(UAU^{-1})$.
\end{theorem}
\begin{theorem}[Wielandt, Davis]
\plabel{thm:DWtrans}
If $f:\lambda\mapsto \frac{a\lambda+b}{c\lambda+d}$, $a,b,c,d\in\mathbb C$, $ad-bc\neq0$ is a complex fractional linear
(i.~e.~complex M\"obius) transformation, and $-\frac dc$ is not in the spectrum of $A$,
then for $f(A)=\frac{aA+b}{cA+d}$, the Davis--Wielandt shell is given as
\[\DW_*(f(A))=f_*(\DW_*(A)).\]
\end{theorem}

\begin{theorem}[Wielandt, Davis]
\plabel{thm:DWfonc}
The asymptotic points in  $\DW_{*}(A)$ are the $\iota_*(\lambda)$ where $\lambda$ is an eigenvalue of $A$.
\end{theorem}
\begin{proof}[Proofs]
See Davis \cite{D1}, \cite{D2}.
(Only Theorem \ref{thm:DWtrans} is nontrivial.)
\end{proof}

\subsection{The elliptical range theorem for the shell
(statement)
}\plabel{ss:DW}
~\\
\begin{theorem}[Wielandt, Davis]
\plabel{thm:DWdonc}
(The qualitative elliptical range theorem for the Davis--Wielandt shell.)

Suppose that $A$ is a linear operator on a $2$-dimensional complex Hilbert space.
We have the following possibilities:

(i) $A$ has a double eigenvalue $\lambda$, and $A$ is normal (thus $A=\lambda \Id$).

Then $\DW_*(A)$ contains only the point $\iota_*(\lambda)$.

(ii)  $A$ has two different eigenvalues $\lambda_1\neq\lambda_2$, and $A$ is normal.

Then $\DW_*(A)$ is the asymptotically closed $h$-line connecting $\iota_*(\lambda_1)$ and $\iota_*(\lambda_2)$.

(iii) $A$ has a double eigenvalue $\lambda$, and $A$ is not normal.

Then $\DW_*(A)$ is an asymptotically closed $h$-horosphere with  asymptotical point $\iota_*(\lambda)$.
In the $\mathrm{BCK}$ model this is an ellipsoid, whose equation is linearly
generated by the quadratic equation of the unit sphere
and the equation of the double plane tangent to unit sphere at $\iota_*(\lambda)$.

(iv) $A$ has two different eigenvalues $\lambda_1\neq\lambda_2$, and $A$ is not normal.

Then $\DW_*(A)$ is the an asymptotically closed $h$-tube around the $h$-line connecting $\iota_*(\lambda_1)$ and $\iota_*(\lambda_2)$.
In the $\mathrm{BCK}$ model this is an ellipsoid, whose equation is linearly
generated by the quadratic equation of the unit sphere
and the quadratic equation of the union of planes tangent to unit sphere at $\iota_*(\lambda_1)$ and $\iota_*(\lambda_2)$.
\end{theorem}

This theorem is not stated in this form by Wielandt or Davis; but it
 contains information known to them, by Wielandt \cite{Wie0}, \cite{Wie}, Davis \cite{D1}, \cite{D2}.
Its proof merely combines the basic information and symmetry principles embodied in the previous theorems,
 and basic facts about analytic hyperbolic geometry.
More explicitly, it is demonstrated in \cite{L2}, Appendix B.

The statement above is qualitative only, because it lacks those hyperbolic metric informations
 which would specify the concrete hyperbolic quadrics.
For this reason, we will \textit{reprove} Theorem \ref{thm:DWdonc} with the following metric addendum:
\begin{theorem}[Addendum to Theorem \ref{thm:DWdonc}, on the inner metrical data]\plabel{thm:addDW1}
~

In the cases (ii)/(iv) the natural hyperbolic radius of the possibly degenerate $h$-tube $\DW_*(A)$
is
\begin{equation}
\frac12\arcosh\frac{U_A}{|D_A|}\equiv\arsinh\sqrt{\frac{U_A-|D_A|}{2|D_A|}}
\equiv\artanh\sqrt{\frac{U_A-|D_A|}{U_A+|D_A| }};
\plabel{eq:addDW1}
\end{equation}
\begin{commentx}
\[
\equiv\arcosh\sqrt{\frac{U_A+|D_A|}{2|D_A|}}
\]
\end{commentx}
where the quantities $U_A$, $|D_A|$ are defined in Subsection \ref{ssub:spectype}.

In cases (i)/(iii)
the formula \eqref{eq:addDW1} gives only the values $0$ (by declaration) and $+\infty$ respectively.
\end{theorem}
In the cases (i), (ii), (iv) the theorem characterizes  $\DW_*(A)$.
However, it does not characterize the horospheres of case (iii).
A lazy way around  this is
\begin{theorem}[Addendum to Theorem \ref{thm:DWdonc}, using ``external'' data]\plabel{thm:addDW2}
~

If the $A$ has eigenvalues $\lambda_1$, $\lambda_2$, and operator norm $\|A\|$, then
$\DW_*(A)$ is characterized as the possibly degenerate asymptotically closed
$h$-tube or $h$-horosphere, whose asymptotic points are  $\iota_*(\lambda_1)$ and $\iota_*(\lambda_2)$,
and it touches the plane $z_{\mathrm{pCK}}=\|A\|^2$ or $z_{\mathrm{BCK}}=\frac{\|A\|^2-1}{\|A\|^2+1}$
(from below), respectively, depending on the model used.
\end{theorem}
All these statements will be proven in Section \ref{sec:DWP}.
\snewpage

\section{The conformal range}
\plabel{sec:CR}

\subsection{On the hyperbolic plane
(review)
}~\\

\textbf{Complex parametrization and collineations.}
Regarding models  BCK, pCK, let
 $\pi^{[2]}_{*}:\overline{H^3_{*}}\rightarrow \overline{H^2_{*}}$ be the canonical projection
 from the hyperbolic $3$-space to the hyperbolic $2$-space
 by the elimination of the second coordinate, or, in other terms, setting it to $0$.
Geometrically, this is an $h$-orthogonal projection from the $h$-space $\overline{H^3_{*}}$ to the canonically embedded
 (asymptotically closed) $h$-plane $i^{[2]}(\overline{H^2_{*}})$.
~
Let $\iota^{[2]}_*=\pi^{[2]}_*\circ\iota_*$, the mapping of the Riemann-sphere to the asymptotically closed hyperbolic plane.
Then, for $\lambda\in\mathbb C$
\begin{equation}
\iota_{\mathrm{pCK}}^{[2]}(\lambda)=\left(\Rea\lambda,|\lambda|^2\right)
\quad\,\,\text{and}\quad\,\,
\iota_{\mathrm{BCK}}^{[2]}(\lambda)=\left(\frac{2\Rea\lambda}{|\lambda|^2+1},
\frac{|\lambda|^2-1}{|\lambda|^2+1}\right);
\plabel{eq:jimago}
\end{equation}
 and $\iota_{\mathrm{pCK}}^{[2]}(\infty)=\infty_{\langle0,1\rangle}$, $\iota_{\mathrm{BCK}}^{[2]}(\infty)=(0,1)$.
The map $\iota^{[2]}_*$ is onto; one-to-one from $\mathbb R\cup\{\infty\}$ onto the asymptotic points,
 and two-to-one from $\mathbb C\setminus\mathbb R$ to non-asymptotic points.

\textit{All} collineations (i.~e.~congruences, i.~e.~isometries) of the hyperbolic plane (not only just orientation-preserving ones)
 are induced by real fractional linear
 transformations $f:\lambda\mapsto \frac{ax+b}{cx+d}$ where $a,b,c,d\in\mathbb R$, $ad-bc\neq0$;
 such that $f_*$ restricts to our canonical $h$-plane.
Regarding that restriction, it is useful to know that for real M\"obius transformations $f$,
 the effects $f_*$ commute with  $\pi^{[2]}_*$.
Consequently,
\begin{equation}
f_*(\iota^{[2]}_*(\lambda) )=\iota^{[2]}_*(f(\lambda) ).
\plabel{eq:imago}
\end{equation}
(Here $\sgn(ad-bc)$ shows whether the $h$-collineation is orientation-preserving or not.)

As we have seen, the Riemann-sphere up to conjugation can be used to parametrize
 the asymptotically closed hyperbolic space.
In practice the part of the Riemann-sphere with nonnegative imaginary part is used for this purpose:
This is essentially the Poincar\'e half-plane model:
Its base set is $\overline{H}^2_{\mathrm{Ph}}=\{z\in\mathbb C\,:\,\Rea z\geq0\}\cup\{\infty\}$.
Here $\mathbb C$ is routinely identified with $\mathbb R^2$.
The canonical transcription to the projective plane models can be abstracted from \eqref{eq:jimago}; and
 the canonical transcription from the projective plane models to the Poincar\'e half-plane model $\mathrm{Ph}$ is given by
\begin{align*}
(x_{\mathrm{Ph}},z_{\mathrm{Ph}})
&=\left(x_{\mathrm{pCK}},\sqrt{z_{\mathrm{pCK}}-(x_{\mathrm{pCK}})^2 }   \right)
=\frac{(x_{\mathrm{BCK}},\sqrt{1-(x_{\mathrm{BCK}})^2-(z_{\mathrm{BCK}})^2 })}{1-z_{\mathrm{BCK}}}          .
\end{align*}
In terms of the complex parametrization, this yields (in accordance to the original idea)
\[\iota_{\mathrm{Ph}}^{[2]}(\lambda)=\left(\Rea\lambda,|\Ima\lambda| \right),\qquad \iota_{\mathrm{Ph}}^{[2]}(\infty)=\infty.\]
As \eqref{eq:imago} remains valid for Ph,
 the collineation effect for $w=x_{\mathrm{Ph}}+\mathrm iz_{\mathrm{Ph}}$ is given by
\[ f_{\mathrm{Ph}}^{[2]}: w\mapsto \left(\frac{aw+b}{cw+d}\right)^{\text{ conjugated if } ad-bc<0}.\]

We could have left out the Poincar\'e half-plane model by talking only about the parametrization by $\iota^{[2]}$,
 but there is little sense in hiding it.
It is well-known that $h$-lines in the Ph model look
like half-lines and half-circles perpendicular to the real line.

\snewpage

\textbf{Distance.}
In the BCK model,  the natural distance function  is given by
\begin{multline}\mathrm d^{\mathrm{BCK}}((x_{\mathrm{BCK}}, z_{\mathrm{BCK}}),(\tilde x_{\mathrm{BCK}}, \tilde z_{\mathrm{BCK}}))=\\
=\arcosh\left(\frac{1-x_{\mathrm{BCK}}\tilde x_{\mathrm{BCK}} - z_{\mathrm{BCK}}\tilde z_{\mathrm{BCK}}  }{\sqrt{1-(x_{\mathrm{BCK}})^2-(z_{\mathrm{BCK}})^2}\sqrt{1-(\tilde x_{\mathrm{BCK}})^2-(\tilde z_{\mathrm{BCK}})^2}}\right).
\plabel{eq:dis2}
 \end{multline}
In particular,
\begin{align}
\mathrm d^{\mathrm{BCK}}( (0,0),(s,0))
&=\arcosh\frac{1}{\sqrt{1-s^2}}
=\arsinh\frac{|s|}{\sqrt{1-s^2}}
=\artanh |s|
\plabel{eq:distex2}
\\
\notag
&=\frac12\arcosh\frac{1+s^2}{1-s^2}
=\frac12\arsinh\frac{2|s|}{1-s^2}
=\frac12\artanh\frac{2|s|}{1+s^2}
.
\end{align}
In the Ph model, the natural distance function  is given by
\begin{align}
\mathrm d^{\mathrm{Ph}}((x_{\mathrm{Ph}} ,z_{\mathrm{Ph}}),(\tilde x_{\mathrm{Ph}} ,\tilde z_{\mathrm{Ph}}))&
=\arcosh\left( 1+\frac{(x_{\mathrm{Ph}}-\tilde x_{\mathrm{Ph}})^2 + (z_{\mathrm{Ph}}-\tilde z_{\mathrm{Ph}})^2}{2z_{\mathrm{Ph}}\tilde z_{\mathrm{Ph}}}\right)\plabel{eq:dis2Ph}\\
&=2\arsinh\left( \frac{\sqrt{(x_{\mathrm{Ph}}-\tilde x_{\mathrm{Ph}})^2 +(z_{\mathrm{Ph}}-\tilde z_{\mathrm{Ph}})^2}}{2\sqrt{z_{\mathrm{Ph}}\tilde z_{\mathrm{Ph}}}}\right).\notag
\end{align}
The latter formula shows well how the Euclidean metric deforms in the Ph model.
\snewpage

\textbf{Representative for horocycles.}
As in the case of hyperbolic space, we can argue
\begin{equation}
z_{\mathrm{pCK}}-(x_{\mathrm{pCK}})^2=c,
\plabel{eq:choro1}
\end{equation}
with $c>0$ will be the $h$-horocycles in the pCK model with asymptotic point $\infty_{\langle0,1\rangle}$.
Accordingly,
\begin{equation}
(x_{\mathrm{BCK}})^2+(z_{\mathrm{BCK}})^2-1+c(z_{\mathrm{BCK}}-1)^2=0
\plabel{eq:chorom}
\end{equation}
with $c>0$ will be the $h$-horocycles in the BCK model with asymptotic point $(0,1)$.
(They are linearly generated by the absolute $(x_{\mathrm{BCK}})^2+(z_{\mathrm{BCK}})^2-1=0$ and
the double tangent $(z_{\mathrm{BCK}}-1)^2=0$ at $(0,1)$.)
In the Ph model, our $h$-horosphere has equation
\[z_{\mathrm{Ph}}=\sqrt c,\]
as it is easy to see from \eqref{eq:choro1}.
In fact, such an equation for an $h$-horocycle should be quite familiar to those with any experience with the Ph model.

In the case $c=1$ the  respective equations are
\begin{equation}
(x_{\mathrm{BCK}})^2+2(z_{\mathrm{BCK}})^2-2z_{\mathrm{BCK}}=0
\plabel{eq:choro}
\end{equation}
and
\[z_{\mathrm{Ph}}=1.\]
~
Let us call this $h$-horocycle $\tilde C$.
By symmetry reasons, one can see that the map
\[(x_{\mathrm{Ph}}, z_{\mathrm{Ph}})\mapsto(x_{\mathrm{Ph}},1)\]
will be an orthogonal (pedal) projection to $\tilde C$.
Therefore, the signed distance from $\tilde C$ (negative inside the horocycle) will be given as
\begin{align*}
\mathrm d_{\tilde C}^*(x_*,z_*)&=\sgn(1-z_{\mathrm{Ph}})\mathrm d_{\tilde C}^{\mathrm{Ph}}
\left( (x_{\mathrm{Ph}} ,z_{\mathrm{Ph}} ), (x_{\mathrm{Ph}} ,1 )\right)\\
&=-\log z_{\mathrm{Ph}}\\
&=\log\left(\frac{1-z_{\mathrm{BCK}} }{\sqrt{1-(x_{\mathrm{BCK}})^2-(z_{\mathrm{BCK}})^2}}\right).
\end{align*}
One can see that the level sets of $\mathrm d_{\tilde C}$ are horocycles themselves (the parallel horocycles).
The level set with value $d$ corresponds to the earlier horocycles with $c=\exp(-2d)$.

\textbf{Representatives for $h$-distance lines.}
Similarly to the case of $h$-tubes, it yields
\begin{equation}
(1+c)(x_{\mathrm{BCK}})^2+ (z_{\mathrm{BCK}})^2-1 =0 ,
\plabel{eq:chorocc}
\end{equation}
with $c>0$, as the equations for asymptotically closed
 $h$-distance lines belonging to the $h$-line $x_{\mathrm{BCK }}=0$;
 that is with asymptotical points $(0,\pm1)$ in the BCK model.
~
(For the analytically minded, we remark that, in contrast to the $h$-tubes, \eqref{eq:chorocc} is
 better to be interpreted as linear, in fact, positive, combination, of the equations
 $(x_{\mathrm{BCK}})^2+(z_{\mathrm{BCK}})^2-1=0$ for the absolute and the double line $(x_{\mathrm{BCK}})^2=0$
 for the axis.)

\snewpage
~

\subsection{The conformal range}
~\\

\textbf{Basic properties.} As we have discussed before,
\begin{equation}
\DW_{*}^{\mathbb R}(A)=\pi^{[2]}_{*} (\DW_{*}(A)).
\plabel{eq:CRproj}
\end{equation}

We remark that regarding the Poincar\'e half-plane model,
\[\DW_{\mathrm{Ph}}^{\mathbb R}(A)=
\left\{
\left(\frac{\Rea\langle \mathbf y,\mathbf x\rangle}{|\mathbf x|_2^2}  ,
\frac{\sqrt{|\mathbf x|_2^2|\mathbf y|_2^2-|\Rea \langle\mathbf y,\mathbf x\rangle|^2}}{|\mathbf x|_2^2}\right)
\,:\, A\mathbf x=\mathbf y, \, \mathbf x\neq 0
\right\}
\subset
 \overline{H}^2_{\mathrm{Ph}}.\]

\begin{theorem}
\plabel{thm:CRronc}
If $U$ is unitary, then  $\DW^{\mathbb R}_{*}(A)=\DW^{\mathbb R}_{*}(UAU^{-1})$.
\end{theorem}

 \begin{theorem}
\plabel{thm:CRtrans}
If $f:\lambda\mapsto \frac{a\lambda+b}{c\lambda+d}$, $a,b,c,d\in\mathbb R$, $ad-bc\neq0$ is a real fractional linear
(i.~e.~real M\"obius) transformation, and $-\frac dc$ is not in the spectrum of $A$,
then for $f(A)=\frac{aA+b}{cA+d}$, the conformal range is given as
\[\DW^{\mathbb R}_*(f(A))=f_*(\DW^{\mathbb R}_*(A)).\]
\end{theorem}

\begin{theorem}
\plabel{thm:CRfonc}
The asymptotic points in  $\DW^{\mathbb R}_{*}(A)$ are the $\iota^{[2]}_*(\lambda)$ where $\lambda$ is a real eigenvalue of $A$.
(But any eigenvalue $\lambda$ implies $\iota^{[2]}_*(\lambda)\in\DW^{\mathbb R}_{*}(A)$.)
\end{theorem}
In what follows, the $h$-eigenpoints of $A$ are the $\iota^{[2]}_*(\lambda)$,
where $\lambda$ is an eigenvalue of $A$.

\begin{proof}[Proofs]
These immediately follow from the corresponding theorems of Subsection \ref{ss:DWbas},
taking the properties of $\pi^{[2]}_{*}$ into account.
\end{proof}
Theorems \ref{thm:CRronc}--\ref{thm:CRfonc} are also naturally valid for $*=\mathrm{Ph}$.

\begin{remark}
\plabel{lem:real}
The real conformal range makes sense directly for operators acting on real Hilbert spaces.
In the real case, Theorem \ref{thm:CRronc} applies to orthogonal $U$.
Otherwise the theorems above remain valid.
In general, if $A^{\mathbb R}$ acts on a real Hilbert space then its conformal range
 then its conformal range is same as the conformal range of its canonical complexification $A^{\mathbb C}$,
 except possibly in the case when $A^{\mathbb R}$ is $2$-dimensional:

In the 2-dimensional case:
One can see (cf.~\cite{L2}) that
the conformal range
$\DW_{*}^{\mathbb R}(A^{\mathbb R})$ is either a $h$-circle, an $h$-horocycle, or a pair of $h$-hypercyles
or a degenerate version of these;
corresponding to the elliptic, parabolic, and hyperbolic cases for $A^{\mathbb R}$, respectively.
But considering the shell $\DW_{*}^{\mathbb R}(A^{\mathbb C})$, these cycles are filled in.
In terms of the Davis--Wielandt shell  $\DW_{*}(A^{\mathbb C})$, these correspond either to the case
of an $h$-tube with axis perpendicular to $i^{[2]}(\overline{H^2_{*}})$, or
an $h$-horosphere with asymptotic point on $i^{[2]}(\overline{H^2_{*}})$,
or  an $h$-tube with axis lying on $i^{[2]}(\overline{H^2_{*}})$ respectively.
\qedremark
\end{remark}
\snewpage

\subsection{The qualitative elliptical range theorem for the conformal range}\plabel{ss:CR}
~\\

A more detailed  qualitative description of the conformal range of $2\times2$ matrices
is given by:

\begin{theorem}
\plabel{thm:CRdonc}

Suppose that $A$ is a linear operator on a $2$-dimensional complex Hilbert space.
We have the following possibilities:

(i) $A$ has a double non-real eigenvalue $\lambda$, or two conjugate non-real eigenvalues $\lambda=\bar\lambda$, and $A$ is normal.

Then $\DW^{\mathbb R}_*(A)$ contains only the ordinary $h$-point $\iota^{[2]}_*(\lambda)$.

(ii) $A$ has two not equal, nor conjugate non-real eigenvalues $\lambda_1,\lambda_2$, and $A$ is normal.

Then $\DW^{\mathbb R}_*(A)$ is the $h$-segment connecting  $\iota^{[2]}_*(\lambda_1)$ and $\iota^{[2]}_*(\lambda_2)$.

(iii)  $A$ has two different real eigenvalues $\lambda_1,\lambda_2$, and $A$ is normal.

Then $\DW^{\mathbb R}_*(A)$ is the asymptotically closed $h$-line connecting $\iota^{[2]}_*(\lambda_1)$ and $\iota^{[2]}_*(\lambda_2)$.

(iv) $A$ has a non-real eigenvalue $\lambda_1$ and a real eigenvalue $\lambda_2$, and $A$ is normal.

Then $\DW^{\mathbb R}_*(A)$ is the asymptotically closed $h$-half line connecting  $\iota^{[2]}_*(\lambda_1)$ and $\iota^{[2]}_*(\lambda_2)$.

(v) $A$ has a double real eigenvalue $\lambda$, and $A$ is normal.

Then $\DW^{\mathbb R}_*(A)$ contains only the asymptotic $h$-point $\iota^{[2]}_*(\lambda)$.

(vi) $A$ has a double non-real eigenvalue $\lambda$, or two conjugate non-real eigenvalues $\lambda=\bar\lambda$, and $A$ is not normal.

Then $\DW^{\mathbb R}_*(A)$ is an $h$-circle.
Here $\iota^{[2]}_*(\lambda)$ is an interior point.

(vii) $A$ has two not equal, nor conjugate non-real eigenvalues $\lambda_1,\lambda_2$, and $A$ is not normal.

Then $\DW^{\mathbb R}_*(A)$ is a proper $h$-ellipse.
Here $\iota^{[2]}_*(\lambda_1)$ and $\iota^{[2]}_*(\lambda_2)$ are interior points.

(viii)  $A$ has two different real eigenvalues $\lambda_1,\lambda_2$, and $A$ is not normal.

Then $\DW^{\mathbb R}_*(A)$ is a $h$-distance band around the $h$-line connecting $\iota^{[2]}_*(\lambda_1)$ and $\iota^{[2]}_*(\lambda_2)$.

(ix) $A$ has a non-real eigenvalue $\lambda_1$ and a real eigenvalue $\lambda_2$, and $A$ is not normal.

Then $\DW^{\mathbb R}_*(A)$ is an $h$-elliptic parabola with asymptotic point $\iota^{[2]}_*(\lambda_2)$.
Here $\iota^{[2]}_*(\lambda_1)$ is an interior point.

(x) $A$ has a double real eigenvalue $\lambda$, and $A$ is not normal.

Then $\DW^{\mathbb R}_*(A)$ is a $h$-horodisk with asymptotic point $\iota^{[2]}_*(\lambda)$.

\end{theorem}

As much as Theorem \ref{thm:CRdonc} can be proven relatively easily from Theorem  \ref{thm:DWdonc} and \eqref{eq:CRproj};
 except case (vii) requires some thinking in order to show that it cannot be an $h$-circle;
 and case (ix)  that it cannot be an $h$-horodisk.
The problem is that in cases (vi), (vii), (ix), it does not say much
 about the interior $h$-eigenpoints $\iota_*^{[2]}(\lambda)$,
 and it also lacks the further metric information in the non-normal cases (vi)--(x),
 which would specify the conics (conical disks) concretely.
For these reason, we prove Theorem \ref{thm:CRdonc} in Section \ref{sec:CRP} with the following addendum:
\snewpage
\begin{theorem}[Addendum to Theorem \ref{thm:CRdonc}, on the inner metrical data]
\plabel{thm:addCR1}
~

(a) In the cases  (i)/(vi), (ii)/(vii), and (iv)/(ix)
the non-asymptotic $h$-eigenpoints act as foci for the synthetic presentation of the
possibly degenerate $h$-ellipses and $h$-elliptic parabolas respectively.
If the asymptotic points are interpreted as foci, then this extends to all cases.

(b) In the cases  (i)--(iv)/(vi)--(ix), the (generalized) minor semi-axis length is
\begin{equation}
s^-=\frac12\arcosh\left(1+\frac{U_A-|D_A| }{\max(E_A,|D_A|)}\right)\equiv\arsinh \sqrt{\frac{U_A-|D_A| }{2\max(E_A,|D_A|)}};
\plabel{eq:CRminor}
\end{equation}
\begin{commentx}
\[\equiv \artanh \sqrt{\frac{U_A-|D_A| }{\max(U_A+|D_A|, U_A-|D_A|+2E_A   )}}\]
\end{commentx}
where $U_A$, $|D_A|$ are defined in Subsection \ref{ssub:spectype},
and $E_A$ is defined in Subsection \ref{ssub:squant}.

In the cases (v) and (x), \eqref{eq:CRminor} would yield the values $0$ (by declaration) and $+\infty$, respectively.
\end{theorem}
\begin{theorem}[Continuation of Theorem \ref{thm:addCR1}] \plabel{note:addCR1}
~

In similar manner, the (generalized) major semi-axes length is
\[s^+=\frac12\arcosh\frac{U_A+E_A}{\left|E_A-|D_A|\right|};\]
\begin{commentx}
\[ \equiv\frac12\arcosh\left(\frac{E_A+|D_A|}{\left|E_A-|D_A|\right|}+\frac{U_A-|D_A|}{\left|E_A-|D_A|\right|} \right)
\equiv \artanh \sqrt{\frac{ \min(U_A+|D_A|, U_A-|D_A|+2E_A   ) }{\max(U_A+|D_A|, U_A-|D_A|+2E_A   )}} \]
\end{commentx}
and the half of the (generalized) focal distance is
\[s^{\mathrm f}=\frac12\arcosh\frac{E_A+|D_A|}{\left|E_A-|D_A|\right|}
. \]
\begin{commentx}
\[ \equiv \artanh \sqrt{\frac{\min(U_A+|D_A|, U_A-|D_A|+2E_A   )-(U_A-|D_A|)  }{\max(U_A+|D_A|, U_A-|D_A|+2E_A   )-(U_A-|D_A|)}} \]
\end{commentx}
\end{theorem}

Theorem \ref{thm:addCR1} is sufficient to characterize the conformal range in cases (i)--(ix).
However, it is insufficient  in case (x), when the conformal range is a horodisk.
Again, the  lazy way to achieve that is by
\begin{theorem}[Addendum to Theorem \ref{thm:CRdonc}, using ``external'' data]\plabel{thm:addCR2}
~

If the $A$ has eigenvalues $\lambda_1$, $\lambda_2$, and operator norm $\|A\|$, then
$\DW_*(A)$ is characterized as the possibly degenerate $h$-ellipse or  asymptotically closed
$h$-elliptical parabola, whose generalized foci are  $\iota_*(\lambda_1)$ and $\iota_*(\lambda_2)$,
and it touches the line $z_{\mathrm{pCK}}=\|A\|^2$ or $z_{\mathrm{BCK}}=\frac{\|A\|^2-1}{\|A\|^2+1}$
(from below), respectively, depending on the model used.
\end{theorem}
All these statements will be proven in Section \ref{sec:CRP}.

\snewpage
\section{On $2\times2$ complex matrices}
\plabel{sec:spec22}

Here all matrices will be $2\times2$ complex matrices.
\subsection{The spectral and metric discriminants}
\plabel{ssub:spectype}
~\\

We use the notation
\begin{equation}
D_A=\det\left( A-\frac{\tr A}2\Id_2\right)=(\det A)-\frac{(\tr A)^2}4.
\plabel{nt:DA}
\end{equation}
It is essentially the discriminant of $A$, as the eigenvalues of $A$ are $\frac12\tr A\pm\sqrt{-D_A}$.

In terms of trace,
\begin{equation*}-D_A
=\frac12\tr\left(\left(A-\frac{\tr A}{2}\Id\right)^2\right)
=\frac{\tr A^2}2-\frac{(\tr A)^2}4
.\end{equation*}
Thus, a similar notion is given by the metric discriminant
\begin{equation}U_A
=\frac12\tr\left(\left(A-\frac{\tr A}{2}\Id\right)^*\left(A-\frac{\tr A}{2}\Id\right)\right)
=\frac{\tr A^*A}2-\frac{|\tr A|^2}4
.\plabel{nt:UA}\end{equation}
(We will see, however, that $U_A$ is not truly a counterpart of $-D_A$, but of $|D_A|$.)

It is elementary, but useful to keep in mind that any $2\times2$ complex matrix $A$ can be brought
by unitary conjugation into from
\begin{equation}\begin{bmatrix}
\lambda_1&t\\0&\lambda_2
\end{bmatrix}
\plabel{eq:canon2}\end{equation}
where $t\geq0$.
This form is unique, up to the order of the eigenvalues $\lambda_1$ and $\lambda_2$.

In the situation above,
\begin{equation}U_A=\frac{t^2}2+\left|\frac{\lambda_1-\lambda_2}2\right|^2,
\qquad\text{and}\qquad D_A=-\left(\frac{\lambda_1-\lambda_2}2\right)^2.
\plabel{eq:cansect}
\end{equation}

Conversely,
\begin{equation}
\begin{matrix}
\text{$\lambda_1,\lambda_2$\qquad are the two roots of}
\qquad
\lambda^2-(\tr A)\lambda+\det A=0;
\end{matrix}
\plabel{eq:canroots}\end{equation}
and the ``canonical off-diagonal quantity'',  as \eqref{eq:cansect} shows, is
\begin{equation}t=\sqrt{2(U_A-|D_A|)}.\plabel{eq:canoffdiag}\end{equation}
(One familiar with the elliptical range theorem
can recognize that $U_A$ and $|D_A|$ are quantities to play role there.
In particular, \eqref{eq:canoffdiag} is the length of the minor axis.)

\begin{lemma}\plabel{lem:UADA}
Suppose that $A$ is a $2\times2$ complex matrix. Then
\[U_A\geq |D_A|.\]
If $A$ is normal, then $U_A= |D_A|$.
If $A$ is not normal, then $U_A> |D_A|$.
\begin{proof}
Using a conjugation by a unitary matrix (which leaves $U_A$, $|D_A|$ invariant), we can assume
that $A$ is of triangular form \eqref{eq:canon2}.
Then \eqref{eq:cansect} shows the statement.
\end{proof}
\begin{proof}[Alternative proof]
From \eqref{nt:UA}
we already see that $U_A\geq 0$.
By direct computation,
\begin{commentx}
(as a consequence of
$\tr [A,\mathbf v]^2=8\left(\left(\frac12\tr
  \left(\left(A-\frac{\tr A}2\Id_2\right)\left(\mathbf v-\frac{\tr \mathbf v}2\Id_2\right)\right)\right)^2-D_AD_{\mathbf v}\right)$ )
\end{commentx}
\[\tr (A^*A-AA^*)^*(A^*A-AA^*)=8\left((U_A)^2-|D_A|^2\right).\]
The LHS is the squared Frobenius norm of the commutator $[A^*,A]$.
This is nonnegative, and vanishes if and only if $A$ is normal.
\end{proof}
\end{lemma}

If $f:\lambda\mapsto \frac{a\lambda+b}{c\lambda+d}$, $a,b,c,d\in\mathbb C$, $ad-bc\neq0$ is a complex fractional linear
(i.~e.~complex M\"obius) transformation, and $A$ is a complex $2\times2$ matrix, and $-\frac dc$ is not in the spectrum of $A$,
then we say that $f$ is applicable to $A$. (Indeed, we can take $f(A)$.)
\begin{lemma}\plabel{lem:UDconf}
For complex $2\times 2$ matrices $A$, the ratio
\[U_A : |D_A|\]
is invariant under complex M\"obius transformations (whenever they are applicable to $A$).
\begin{proof}
(o) Multiplying $A$ by a nonzero scalar $a$ yields
$U_{a A}= \bar aa U_A=|a|^2 U_A $ and $|D_{a A}|=|a^2 D_A|=|a|^2|D_A|$;
therefore  $U_{\alpha A}:|D_{\alpha A}|= U_A : |D_A|$, thus the ratio remains invariant in this case.
(i) adding a scalar matrix to $A$ does not alter either ${U_A}$ or ${|D_A|}$; thus the ratio remains the same.
(ii) Regarding taking the inverse of $A$ (if applicable),
\[{U_{A^{-1}}}:{|D_{A^{-1}}|}\equiv{\dfrac{U_A}{|\det A|^2}}:{\left|\dfrac{D_A}{(\det A)^2}\right|}={U_A}:{|D_A|}\]
holds.
The invariance properties (o), (i), (ii) together   imply the complex M\"obius invariance property.
\begin{commentx}
Remark: (o) is already implied by (i) and (ii).
Let $\gamma\neq0$ be a scalar. Let $\delta$ be an arbitrarily chosen scalar but so that
 $A$ should have no eigenvalue $\delta$ or $\delta+1/\gamma$.
 Then the equality
\begin{equation}
-\gamma^2A=\left(\left(\left(A-\delta\Id_2-\frac1\gamma\Id_2\right)^{-1}+\gamma\Id_2\right)^{-1}- \frac1\gamma\Id_2\right)^{-1}
-\delta\gamma^2\Id_2
\plabel{eq:conftrick}
\end{equation}
implies invariance for the scalar multiplier $a=-\gamma^2$.
\end{commentx}
\end{proof}
\end{lemma}

One can also think that this quantity, or any of its variants,
$\frac{U_A}{|D_A|}$,  $\frac{|D_A|}{U_A}$, or $\frac{U_A-|D_A|}{U_A+|D_A|}$,
measure non-normality in a complex M\"obius invariant way.
(This foreshadows the hyperbolic geometric interpretation.)

\subsection{The canonical representatives (complex M\"obius)}
\plabel{ssub:preDW}
~\\

\begin{lemma}\plabel{lem:preDW}
For  complex $2\times2$ matrices
up to complex M\"obius  transformations and unitary conjugations
(which, of course, commute with each other)
it is sufficient to consider the following representatives:
\begin{equation}
\mathbf 0_2=\begin{bmatrix}0&\\&0\end{bmatrix}
\plabel{eqx:022H}
\end{equation}
(parabolic, normal) ;
\begin{equation}
S_0=\begin{bmatrix} 0&1 \\&0\end{bmatrix}
\plabel{eqx:superbolicH}
\end{equation}
(parabolic non-normal);
\begin{equation}
L_{t}=\begin{bmatrix} 1& 2t\\
&-1 \end{bmatrix} \qquad  t\geq 0
\plabel{eqx:loxodromicH}
\end{equation}
(non-parabolic case; $t=0$: normal, $t>0$: non-normal).

The representatives above are inequivalent.

\begin{proof}
It is sufficient to normalize the eigenvalues by applying M\"obius transformations,
and then achieve normal form by unitary conjugation.
By this, we arrive to the parabolic cases
\begin{equation}
\widehat 0_t=\begin{bmatrix}0&2t\\&0\end{bmatrix} \qquad t\geq0,
\plabel{eqx:nullpreH}
\end{equation}
and the non-parabolic cases \eqref{eqx:loxodromicH}.
The cases $\widehat 0_t$ with $t>0$ scale into each other linearly (hence, by M\"obius transformations).
Thus \eqref{eqx:nullpreH} reduces only to  \eqref{eqx:022H} and \eqref{eqx:superbolicH}.
The ratio $U_A:|D_A|$  distinguishes the various representatives, cf. the following example.
\end{proof}
\end{lemma}
\begin{example}\plabel{ex:preDW}
(a) Regarding $L_t$:
\[U_{L_t}  =1+2t^2, \qquad\text{and}\qquad |D_{L_t}| = 1.\]

(b) Regarding $S_0$:
\[U_{S_0}  =\frac12, \qquad\text{and}\qquad |D_{S_0}| = 0.\]

(c) Regarding $\mathbf 0_2$:
\[U_{\mathbf 0_2}  =0, \qquad\text{and}\qquad |D_{\mathbf 0_2}| = 0.\eqedexer\]
\end{example}
\begin{cor} \plabel{cor:cconftriple}
The ratio
\[
U_A\,\,:\,\,|D_A|
\]
is a full invariant of $2\times2$ complex matrices with respect
to equivalence by complex M\"obius transformations and unitary conjugation.
\begin{proof}
It distinguishes the canonical representatives of Lemma \ref{lem:preDW}.
\end{proof}
\end{cor}
Having
$U_A$ and $|D_A|$, and the crude parabolic / non-parabolic spectral classification, one can
 deal with the elliptical range theorem for the Davis--Wielandt shell.
For sake of the conformal range, a finer description is needed:

\subsection{Spectral type refined}
\plabel{ssub:spectype2}
~\\

We extend the elliptic/parabolic/hyperbolic classification
of real $2\times2$ matrices to complex $2\times2$ matrices as follows.
We distinguish the following classes:

$\bullet$ real-elliptic case: two conjugate, strictly complex eigenvalues, 

$\bullet$ real-parabolic case: two equal real eigenvalues, 

$\bullet$ real-hyperbolic case: two distinct real eigenvalues, 

$\bullet$ non-real parabolic case: two equal strictly complex  eigenvalues, 

$\bullet$ semi-real  case: a real and a strictly complex eigenvalue, 

$\bullet$ quasielliptic case:  two non-conjugate
eigenvalues $\lambda_1,\lambda_2$ with $(\Ima \lambda_1)(\Ima\lambda_2)<0$,  

$\bullet$ quasihyperbolic case:  two distinct
eigenvalues $\lambda_1,\lambda_2$ with $(\Ima \lambda_1)(\Ima\lambda_2)>0$.  

One can see that the classes above are closed for conjugation by unitary matrices
and for real fractional linear (i.~e.~M\"obius) transformations, whenever they are applicable.

\subsection{Another spectral quantity}\plabel{ssub:squant}
~\\

A quantity corresponding to $|D_A|$ is
\[E_A=\left(\frac{(\Ima\tr A)}2\right)^2+\frac{|D_A|-\Rea D_A}2.\]
Indeed, if $\lambda_1$ and $\lambda_2$ are the eigenvalues of $A$,
then it is easy to see that
\begin{equation}
|D_A|=\left|\frac{\lambda_1-\lambda_2}2\right|^2
\plabel{eq:daeig}
\end{equation}
and
\begin{equation}
E_A=\left|\frac{\lambda_1-\bar \lambda_2}2\right|^2;
\plabel{eq:eaeig}
\end{equation}
i.~e.~one corresponds to the other after conjugating an eigenvalue.

\begin{lemma}
\plabel{lem:EDinvar}
The ratio
\[{E_A}:{|D_A|}=\left|{\lambda_1-\bar\lambda_2}|^2:|{\lambda_1-\lambda_2}\right|^2\]
is invariant under real fractional transformations (whenever they are applicable to $A$).
\begin{proof}
It is sufficient to prove that they are invariant to
the transformations
(o) $f:\,\lambda \mapsto a\lambda$ $(a\in\mathbb R\setminus\{0\})$,
(i) $f:\,\lambda \mapsto \lambda+b$ $(b\in\mathbb R)$ and,
(ii) $f:\,\lambda \mapsto 1/\lambda$ (whenever applicable);
and these are easy to see.
\begin{commentx}

(Remark: Applying \eqref{eq:conftrick} once or twice, the case (o) can be omitted.)
\end{commentx}
\end{proof}
\end{lemma}

\subsection{A remark on the real double}
\begin{lemma}\plabel{lem:double}
The identity
\begin{equation}
\left(U_{\mathrm D^{\mathbb R}(A)}\right)^2-\left|D_{\mathrm D^{\mathbb R}(A)}\right|^2=2
\left(U_A-|D_A|\right)\left(U_A+|D_A|\right)\left(U_A-|D_A|+2|E_A|\right)
\plabel{eq:double}
\end{equation}
holds.
\begin{proof}
Due to unitary invariance, this is sufficient to check for matrices of shape \eqref{eq:canon2},
where the computation is relatively manageable.
\end{proof}
\end{lemma}
\begin{cor}
\plabel{cor:double}
$A$ is normal if and only if $\mathrm D^{\mathbb R}(A)$ is normal.
\begin{proof}
As we have $U_A\geq |D_A|\geq 0$ and $U_A\geq |D_A|\geq 0$, in general,  we see:
$\mathrm D^{\mathbb R}(A)$ is normal $\Leftrightarrow$
$U_{\mathrm D^{\mathbb R}(A)} -\left|D_{\mathrm D^{\mathbb R}(A)}\right|=0$
$\Leftrightarrow$
\eqref{eq:double} is $0$
$\Leftrightarrow$
$U_A-|D_A|=0$
$\Leftrightarrow$
$A$ is normal.
\end{proof}
\end{cor}

\snewpage

\subsection{The canonical representatives (real M\"obius)}
\plabel{ssub:preCR}

\begin{lemma}\plabel{lem:preCR}
For  complex $2\times2$ matrices
up to real M\"obius transformations and unitary conjugation
it is sufficient to consider the following cases:
\begin{equation}
\mathbf 0_2=\begin{bmatrix}0&\\&0\end{bmatrix}
\plabel{eqx:022}
\end{equation}
(real-parabolic, normal) ;
\begin{equation}
S_\beta=\begin{bmatrix} 0&\cos\beta \\&\mathrm i\sin\beta\end{bmatrix} \qquad \beta\in\left[0,\frac\pi2\right]
\plabel{eqx:superbolic}
\end{equation}
($\beta=0$: real-parabolic non-normal, $0<\beta<\pi/2$: semi-real non-normal, $\beta=\pi/2$: semi-real normal);
\begin{equation}
L_{\alpha,t}^{\pm}=\begin{bmatrix} \cos\alpha+\mathrm i\sin\alpha& 2t\\
&-\cos\alpha\pm\mathrm i\sin\alpha \end{bmatrix} \qquad \alpha\in\left[0,\frac\pi2\right], t\geq 0
\plabel{eqx:loxodromic}
\end{equation}
\[\text{but the cases  $L_{0,t}^+$ and $L_{0,t}^-$ are identical}\]
($\alpha=0$:   real-hyperbolic case,  $0<\alpha<\pi/2$: quasielliptic [$-$] / quasihiperbolic case [$+$],
$\alpha=\pi/2$: real-elliptic [$-$] / non-real parabolic case [$+$];
$t=0$: normal, $t>0$ non-normal).

Apart from the degeneracy  for $\alpha=0$, the representatives above are inequivalent.

\begin{proof}
It is sufficient to follow the image of the eigenvalues (up to conjugation)
in the Poincar\'e half-plane model (with respect to real M\"obius transformations),
i.~e.~the classification of point-pairs in the asymptotically closed hyperbolic plane
up to collineations,
reconsidering the corresponding the possible complex eigenvalues,
and apply the normal form up to unitary conjugation.
By this, we arrive to the cases
\begin{equation}
\widehat 0_t=\begin{bmatrix}0&2t\\&0\end{bmatrix} \qquad t\geq0,
\plabel{eqx:nullpre}
\end{equation}
\begin{equation}
\widehat S_t^{\pm}=\begin{bmatrix} 0&2t \\&\pm\mathrm i\end{bmatrix}\qquad t\geq0,
\plabel{eqx:superbolicpre}
\end{equation}
\begin{equation}
\widehat L_{\alpha,t}^{\pm_1,\pm_2}=\begin{bmatrix} \cos\alpha\pm_1\mathrm i\sin\alpha& 2t\\
&-\cos\alpha\pm_2\mathrm i\sin\alpha \end{bmatrix} \qquad \alpha\in\left[0,\frac\pi2\right], t\geq 0.
\plabel{eqx:loxodromicpre}
\end{equation}
(The most notable case is when have two non-asymptotic $h$-eigenpoints.
Then by an appropriate collineation we can move them to the line $(x_{\mathrm{Ph}})^2+(z_{\mathrm{Ph}})^2=1$.
By a further translation along this line, we can assume the $h$-eigenpoints are in symmetric position with respect to $(0,1)$.
Then considering the corresponding possible eigenvalues and the normal form by unitary conjugation,
 we arrive to the normal forms \eqref{eqx:loxodromicpre}, with $\alpha\in(0,\pi/2]$.
All other cases are simpler.)

Some cases of signs can be eliminated:
Multiplying by $-1$ and applying unitary conjugation,
we see that $\widehat S_t^{-}$ and $\widehat S_t^{+}$ are equivalent;
we eliminate $\widehat S_t^{-}$.
Similarly, taking inverse and applying unitary conjugation,
we see that $\widehat L_{\alpha,t}^{\pm_1,\pm_2}$ and $\widehat L_{\alpha,t}^{\mp_1,\mp_2}$ are equivalent;
we eliminate $\widehat L_{\alpha,t}^{-,\pm_2}$.
Now, $\widehat 0_0$ is just $\mathbf 0_2$. $\widehat 0_t$ with $t>0$ scales into $S_{0}$.
$\widehat S_t^{+}$ scales into $S_\beta$ with $\beta\in(0,\pi/2]$.
$\widehat L_{\alpha,t}^{+,\pm_2}$ serves $L_{\alpha,t}^{\pm_2}$.

From the location of the eigenvalues, it is not hard to see that
representatives from different families are inequivalent.
Inside the families, the ratio $U_A:|D_A|:E_A$ distinguishes the various representatives as the following example shows.
\end{proof}
\end{lemma}

\begin{example}\plabel{ex:confrep}
(a) Regarding $L_{\alpha,t}^\pm$:
\[U_{L^{+}_{\alpha,t}}=2t^2+(\cos\alpha)^2,
\quad
\left|D_{L^{+}_{\alpha,t}}\right|=(\cos\alpha)^2,
\quad
E_{L^{+}_{\alpha,t}}=1,
\quad
E_{L^{+}_{\alpha,t}}\geq \left|D_{L^{+}_{\alpha,t}}\right|;\]
and
\[U_{L^{-}_{\alpha,t}}=2t^2+1,
\quad
\left|D_{L^{-}_{\alpha,t}}\right|=1,
\quad
E_{L^{-}_{\alpha,t}}=(\cos\alpha)^2,
\quad
E_{L^{-}_{\alpha,t}}\leq\left|D_{L^{-}_{\alpha,t}}\right|.\]

(b) Regarding $S_{\beta}$:
\[U_{S_\beta}=\frac14\left(1+(\cos\beta)^2\right),
\quad
\left|D_{S_\beta}\right|=E_{S_\beta}=\frac14\left(1-(\cos\beta)^2\right).
\]

(c) In case of $\mathbf 0_2$:
\[U_{\mathbf 0_2}=\left|D_{\mathbf 0_2}\right|=E_{\mathbf 0_2}=0.\eqedexer\]
\end{example}
\begin{cor} \plabel{cor:rconftriple}
The triple ratio 
\begin{equation}
U_A\,\,:\,\,|D_A|\,\,:\,\,E_A
\plabel{eq:preconfrep}
\end{equation}
and (when $|D_A|=E_A>0$) the possible choice of
\[\text{real-hyperbolic / semi-real type}\]
together form a full invariant of $2\times2$ complex matrices with respect
to equivalence by real M\"obius transformations and unitary conjugation.
\begin{proof}
\eqref{eq:preconfrep} almost distinguishes the canonical representatives of Lemma \ref{lem:preCR}.
The only ambiguity is between $L_t\equiv L_{0,t}$ and $S_\beta$ when $\cos\beta=\frac{t}{\sqrt{1+t^2}}$.
\end{proof}
\end{cor}
\snewpage
\subsection{The $h$-eigendistance}
\plabel{ssub:eigen}
~\\

\begin{lemma}
\plabel{lem:eigdistinv}
If $f$ is a real M\"obius transformation
applicable to the complex $2\times2$ matrix $A$, then the $h$-eigenpoints of $A$
are related to the $h$-eigenpoints of $f(A)$ by the $h$-isometries associated to $f$.
\begin{proof}
It follows from the effect of $f_*$ on $\iota_*(\lambda)$ and the fact that $f_*$ and $\pi^{[2]}_*$ commutes.
(But it is most obvious in the Poincar\'e half-plane model, where the $h$-eigenpoints are just the eigenvalues up
to conjugation.)
\end{proof}
\end{lemma}
\begin{theorem}
\plabel{lem:eigdist}
If the eigenvalues $\lambda_1,\lambda_2$  of $A$ are from $\mathbb C\setminus\mathbb R$, then
\begin{align}
\mathrm d^*\left(\iota^{[2]}_*(\lambda_1),\iota^{[2]}_*(\lambda_2)\right) &=
\mathrm d^{\mathrm{Ph}}\left((\Rea\lambda_1,|\Ima\lambda_1|),(\Rea\lambda_2,|\Ima\lambda_2|) \right)
\plabel{eq:eigdist}\\
\notag&=
\log\left|\frac{\sqrt{\dfrac{E_A}{|D_A|}}+1}{\sqrt{\dfrac{E_A}{|D_A|}}-1}\right|\equiv
\begin{cases}
2\artanh \sqrt{\dfrac{E_A}{|D_A|}}&\quad\text{if }{\dfrac{E_A}{|D_A|}}\leq1,\\\\
2\arcoth \sqrt{\dfrac{E_A}{|D_A|}}&\quad\text{if }{\dfrac{E_A}{|D_A|}}\geq1
\end{cases}
\\\notag&=\arcosh\left|\frac{E_A+|D_A|}{E_A-|D_A|}\right|.
\end{align}
This is also valid if there is a real eigenvalue but $\lambda_1\neq\lambda_2$.
Then the value in \eqref{eq:eigdist} is $+\infty$.
(Having a double real eigenvalue, a natural choice for the value of  \eqref{eq:eigdist} is $0$.)

\begin{proof}
It can be computed using \eqref{eq:dis2Ph} taking \eqref{eq:daeig} and \eqref{eq:eaeig} into account.
This is a computation entirely in terms of $\Rea\lambda_1,|\Ima\lambda_1|,\Rea\lambda_2,|\Ima\lambda_2|$.
\end{proof}
\end{theorem}
Alternatively, we can argue as follows.
\begin{lemma}
\plabel{lem:eigpoints}
 Regarding the canonical representatives of Lemma \ref{lem:preCR}, their  $h$-eigenpoints in the BCK model are:

For $\mathbf 0_2$:   $\iota_{\mathrm{BCK}}(0)=(0,-1)$ (twice).

For $S_\beta$:   $\iota_{\mathrm{BCK}}(0)=(0,-1)$ and
 $\iota_{\mathrm{BCK}}(\mathrm i\sin\beta)=\left(\frac{\sin\beta)^2-1}{(\sin\beta)^2+1},0\right)$.

For $L_{\alpha,t}$:
 $\iota_{\mathrm{BCK}}(-\cos\alpha\pm\mathrm i\sin\alpha)=(-\cos\alpha,0)$
and $\iota_{\mathrm{BCK}}(\cos\alpha+\mathrm i\sin\alpha)=(\cos\alpha,0)$.
\begin{proof}
Due to the triangular form, the eigenvalues can be recognized immediately.
Computation of the $h$-eigenpoints is straightforward then.
\end{proof}
\end{lemma}

\begin{proof}[Alternative proof for Theorem \ref{lem:eigdist}]
As each side is naturally invariant for real Möbius transformations,
the statement can be reduced for the canonical representatives of Lemma \ref{lem:preCR}.
For example, in the case of non-real eigenvalues this means reduction for $L_{\alpha, t}^\pm$ ($\alpha\in(0,\frac\pi2]$).
Here the $h$-eigenpoints $(\pm\cos\alpha,\sin\alpha)$ are obvious in Ph,
or $(\pm\cos\alpha,0)$ in BCK by Lemma \ref{lem:eigpoints};
 and also the $|D_A|$, $E_A$ by Example \ref{ex:confrep}.
In any case, everything in \eqref{eq:eigdist} computes to
 $\log\frac{1+\cos\alpha}{1-\cos\alpha}=2\artanh\,  \cos\alpha  =\arcosh\frac{1+(\cos\alpha)^2}{1-(\cos\alpha)^2}$.
\end{proof}

It is reasonable to call \eqref{eq:eigdist} `the $h$-eigendistance of $A$'.

\snewpage
\section{The Davis--Wielandt shell of $2\times2$ complex matrices }
\plabel{sec:DWP}

\begin{theorem}\plabel{cor:DW}

The conformal ranges of the canonical representatives of Lemma \ref{lem:preDW} in the BCK model are as follows:

The zero matrix $\mathbf 0_2$ yields the point ellipsoid
\[\{(0,0,-1)\};\]

The matrix $S_0$ yields the rotational ellipsoid with  principal axis
\[\{0\}\times\{0\}\times[-1,0],\]
which is of semi-axis length $\frac12$; and with other semi-axis lengths $\frac{\sqrt2}2$.

The matrix   $L_{ t}$ yields the rotational ellipsoid with principal axis
\[[-1,1]\times\{0\}\times\{0\}\]
which is of semi-axis length $1$; and with other semi-axis lengths $ \frac{t}{\sqrt{1+t^2}}$.
(This ellipsoid is degenerate for $t=0$).

\begin{proof}
Assume that $\mathbf x=z_1\mathbf e_1+z_2\mathbf e_2$ with $|z_1|^2+|z_2|^2=1$.
Then
\begin{multline*}
\bem
\Rea\langle L_t\mathbf x,\mathbf x\rangle \\
\Ima\langle L_t\mathbf x,\mathbf x\rangle \\
 \langle L_t\mathbf x,L_t\mathbf x\rangle \\
 \langle \mathbf x,\mathbf x\rangle
\eem=
   \begin {bmatrix} t&0&-1&0\\ 0&t&0&0
\\ 2t&0&2{t}^{2}&2{t}^{2}+1\\ 0&0&0& 1\end {bmatrix}
\bem
2\Rea( z_2\bar z_1)\\
2\Ima(z_2\bar z_1)\\
|z_2|^2-|z_1|^2\\
|z_2|^2+|z_1|^2
\eem
\\
=
\underbrace{\bem
 {\sqrt{1+t^2}}&&& 0\\
&t&&0 \\
&& {2t}{\sqrt{1+t^2}}&2t^2+1\\
&&&1
\eem}_{S_{Z1}}
   \underbrace{\begin {bmatrix} \frac{t}{\sqrt{1+t^2}}&0&-\frac{1}{\sqrt{1+t^2}}&0\\ 0&1&0&0
\\  \frac{1}{\sqrt{1+t^2}}&0&\frac{t}{\sqrt{1+t^2}}& 0\\ 0&0&0& 1\end {bmatrix}}_{S_{Z2}}
\underbrace{\bem
2\Rea( z_2\bar z_1)\\
2\Ima( z_2\bar z_1)\\
|z_2|^2-|z_1|^2\\
|z_2|^2+|z_1|^2
\eem}_{S_{Z3}}.
\end{multline*}
Here  $S_{Z3}$ ranges over the unit sphere in projective coordinates, but also in affine coordinates;
 being it is technically the same as $\iota_{\mathrm {BCK}}(z_2/z_1)$.
 $S_{Z2}$ leaves it invariant; $S_{Z1}$ shrinks it along the axes, and also translates it along the $z$-axis.
From this, one obtains $\DW_{\mathrm{BCK}}(L_t)$ immediately.
~
Assume $t>0$.  Then $\DW_{\mathrm{BCK}}(L_t)$ is given by the ellipsoid
\[\frac{(x_{\mathrm{pCK}})^2}{\left(\sqrt{1+t^2}\right)^2}+
\frac{\left(y_{\mathrm{pCK}}\right)^2}{t^2} +\frac{\left(z_{\mathrm{pCK}}
-(2t^2+1)\right)^2}{\left(2t\sqrt{1+t^2} \right)^2} =1.\]
Rewriting this according to \eqref{eq:can2}, we obtain the result in the BCK model,
\[\frac{\left(x_{\mathrm{BCK }}\right)^2}{\left(1\right)^2}
+\frac{\left(y_{\mathrm{BCK }}\right)^2}{\left(\dfrac{t}{\sqrt{1+t^2}}\right)^2}
+\frac{\left(z_{\mathrm{BCK }}\right)^2}{\left(\dfrac{t}{\sqrt{1+t^2}}\right)^2}=1.\]
This establishes the statement for $L_t$, $t>0$.
In case $t=0$, the previous argument for  $\DW_{\mathrm{pCK}}(L_0)$
 yields the segment connecting $(-1,0,1)=\iota_{\mathrm{pCK}}(-1)$ and $(1,0,1)=\iota_{\mathrm{pCK}}(1)$.
Transcription to the BCK model yields
 the segment connecting $\iota_{\mathrm{BCK }}(-1)=(-1,0,0)$ and $\iota_{\mathrm{BCK }}(1)=(1,0,0)$,
 which is according to the statement.
Note that the asymptotic points are the same  for $L_t$ with $t>0$.

Let us consider the case of $S_0$.
Here
\[
\bem
\Rea\langle S_0\mathbf x,\mathbf x\rangle \\
\Ima\langle S_0\mathbf x,\mathbf x\rangle \\
 \langle S_0\mathbf x,S_0\mathbf x\rangle \\
 \langle \mathbf x,\mathbf x\rangle
\eem=
   \begin {bmatrix} \frac12&0&0&0\\ 0&\frac12&0&0
\\ 0&0&\frac12&\frac12\\ 0&0&0& 1\end {bmatrix}
\bem
2\Rea( z_2\bar z_1)\\
2\Ima( z_2\bar z_1)\\
|z_2|^2-|z_1|^2\\
|z_2|^2+|z_1|^2
\eem.\]
A similar argument as before  shows that the shell is given by
\[\frac{(x_{\mathrm{pCK}})^2}{\left(\frac12\right)^2}+
 \frac{\left(y_{\mathrm{pCK}}\right)^2}{\left(\frac12\right)^2} +\frac{\left(z_{\mathrm{pCK}}
 -\frac12\right)^2}{\left(\frac12 \right)^2} =1.\]
Rewriting this according to \eqref{eq:can2}, we obtain the result in the BCK model,
\[
\frac{\left(x_{\mathrm{BCK }}\right)^2}{\left(\frac{\sqrt2}2\right)^2}
+\frac{\left(y_{\mathrm{BCK }}\right)^2}{\left(\frac{\sqrt2}2\right)^2}
+\frac{\left(z_{\mathrm{BCK }}+\frac12\right) ^2}{\left(\frac12\right)^2}
=1.\]
This establishes case of $S_0$.
Here the asymptotic point is $\iota_{\mathrm{BCK}}(0)=(0,1)$.

The case of $\mathbf0_2$ is trivial.
\end{proof}
\end{theorem}


\begin{proof}[Proof for Theorem \ref{thm:DWdonc}]

It is sufficient to check the canonical representatives in the various cases:
(i) $\mathbf 0_2$;
(ii) $L_t$ with $t=0$;
(iii) $S_0$;
(iv) $L_t$ with $t>0$.
~
In these cases (cf. also the a the previous proof for equations and interpretations for asymptotic points) the statement
 is mostly clear, the shapes are rather recognizable.
In case (iii), we have a $h$-horosphere:
Note that its equation
\[ (x_{\mathrm{BCK}})^2+ (y_{\mathrm{BCK}})^2+ 2(z_{\mathrm{BCK}})^2+2 z_{\mathrm{BCK}}=0\]
 is just \eqref{eq:horo} but rotated around $y$-axis by $180^\circ$, now with asymptotic point $(0,0,-1)$.
As everything transforms naturally under applicable complex fractional transformations and unitary conjugation,
 we have the statement in general.
In case (iv), we have $h$-tubes:
The equations are like in \eqref{eq:dorocc} but rotated around $y$-axis by $90^\circ$, now with asymptotic points $(\pm1,0,0)$.
\end{proof}

\begin{proof}[Proof for Theorem \ref{thm:addDW1}]
It is sufficient to check the radius in case (iii).
Regarding $L_t$, $t>0$, this radius, by symmetry reasons, is
\[\frac12\mathrm d^{\mathrm{BCK}}\left(\left( 0,0,-\frac{t}{\sqrt{1+t^2}}\right),
\left( 0,0,\frac{t}{\sqrt{1+t^2}}\right)
 \right)=\frac12\arcosh (1+2t^2)=\arsinh t.\]
This matches the statement by Example \ref{ex:preDW}(a).
As things transform invariantly under applicable complex fractional transformations and unitary conjugation,
 we have the statement in general.
\end{proof}

\begin{proof}[Proof for Theorem \ref{thm:addDW2}]
The statement is easier to see in the pCK model.
Having the asymptotic points fixed, it easy to that we have smaller of larger ellipsoids,
 one convex hull is contained in other.
It is also clear that the supremum in $z_{\mathrm{pCK}}$ is the square of the norm,
which strictly grows as the ellipsoids become fatter.
(It ranges from the minimal $\max(|\lambda_1|^2,|\lambda_2|^2)$ toward $+\infty$.)
The transcription to the BCK model is obvious.
\end{proof}
\snewpage
\section{The conformal range of $2\times2$ complex matrices }
\plabel{sec:CRP}

\subsection{Canonical representatives and the geometry of the conformal range}
\plabel{ssub:cangeo}

\begin{theorem}\plabel{cor:CR}

The conformal ranges of the canonical representatives of Lemma \ref{lem:preCR} in the BCK model are as follows:

The zero matrix $\mathbf 0_2$ yields the point ellipse
\[\{(0,-1)\};\]
$S_\beta$ yields the ellipse with axes
\[ \left[-\frac{\sqrt2}{2}\cos\beta,\frac{\sqrt2}{2}\cos\beta \right]\times\left\{-\frac12\right\}\qquad\text{and}\qquad  \{0\}\times[-1,0];\]
and  $L^\pm_{\alpha,t}$ yields the ellipse with axes
\begin{equation*} \left[-\frac{\sqrt{(\cos\alpha)^2+t^2 }}{\sqrt{1+t^2}} ,\frac{\sqrt{(\cos\alpha)^2+t^2 }}{\sqrt{1+t^2}} \right]\times\{0\}\qquad\text{and}\qquad  \{0\}\times\left[-\frac{t}{\sqrt{1+t^2}},\frac{t}{\sqrt{1+t^2}}\right];
\end{equation*}
(the ellipses may be degenerate).
\end{theorem}

\begin{proof}
Assume that $\mathbf x=z_1\mathbf e_1+z_2\mathbf e_2$ with $|z_1|^2+|z_2|^2=1$.
By explicit computation, one finds that
\[
\bem
\Rea\langle L_{\alpha,t}^\pm\mathbf x,\mathbf x\rangle \\ 0 \\ \langle L_{\alpha,t}^\pm\mathbf x,L_{\alpha,t}^\pm\mathbf x\rangle \\ \langle \mathbf x,\mathbf x\rangle
\eem
=
\underbrace{
\bem
t&0&-\cos\alpha&0\\
0&0&0&0\\
2t\cos\alpha&2t\sin\alpha&2t^2&2t^2+1\\
0&0&0&1
\eem
}_{S_T:=}
\bem
2\Rea( z_2\bar z_1)\\
2\Ima( z_2\bar z_1)\\
|z_2|^2-|z_1|^2\\
|z_2|^2+|z_1|^2
\eem.
\]
Now, $S_T$ itself can be written as the product
\begin{multline*}
S_{T1}\cdot S_{T2}\cdot S_{T3}
\equiv
\bem
\sqrt{(\cos\alpha)^2+t^2}&&&\\
&0&&\\
&&2t\sqrt{1+t^2}&2t^2+1\\
&&&1
\eem
\\
\cdot
\bem
1&&&\\
&\dfrac{\sqrt{(\cos\alpha)^2+t^2}}{\sqrt{1+t^2}}&-\dfrac{\sin\alpha}{\sqrt{1+t^2}}&\\\\
&\dfrac{\sin\alpha}{\sqrt{1+t^2}}&\dfrac{\sqrt{(\cos\alpha)^2+t^2}}{\sqrt{1+t^2}}&\\
&&&1
\eem
\cdot
\\
\bem
\dfrac{t}{\sqrt{(\cos\alpha)^2+t^2}}&&-\dfrac{\cos\alpha}{\sqrt{(\cos\alpha)^2+t^2}}&\\
&1&&\\
\dfrac{\cos\alpha}{\sqrt{(\cos\alpha)^2+t^2}}&&\dfrac{t}{\sqrt{(\cos\alpha)^2+t^2}}&\\
&&&1
\eem
\end{multline*}
(if $\cos\alpha=t=0$, then the identity matrix can be taken as $S_{T3}$).
Then it easy to see how $S_T$ acts on the affinized unit sphere:
$S_{T3}$ and $S_{T2}$ are just orthogonal transformations leaving it invariant, while
$S_{T1}$ collapses it to a (possibly degenerate) elliptical disk (and translates it).
This gives $\DW_{\mathrm{pCK}}^{\mathbb R}(L_{\alpha,t}^\pm)$.
~
If $t>0$, then its equation is
\[\frac{(x_{\mathrm{pCK}})^{2}}{(\sqrt{(\cos\alpha)^2+t^2})^2}+\frac{(z_{\mathrm{pCK}}- (2t^2+1))^{2}}{(2t\sqrt{1+t^2})^2}\leq1.\]
Using \eqref{eq:can2},  the transcription
\begin{equation}
\frac{(x_{\mathrm{BCK }})^{2}}{\left(\dfrac{\sqrt{(\cos\alpha)^2+t^2}}{\sqrt{1+t^2}}\right)^2}+
\frac{(z_{\mathrm{BCK }})^{2}}{\left(\dfrac{t}{\sqrt{1+t^2}}\right)^2} \leq1,
\plabel{eq:CRL}
\end{equation}
 to the BCK model is straightforward, yielding the statement.
~
In the $t=0$ case, the set $\DW_{\mathrm{pCK}}^{\mathbb R}(L_{\alpha,0}^\pm)$
 is the segment connecting $(-\cos\alpha ,1)=\iota^{[2]}_{\mathrm{pCK}}(-\cos\alpha\pm \mathrm i\sin\alpha)$
 and $( \cos\alpha, 1)=\iota^{[2]}_{\mathrm{pCK}}(\cos\alpha+\mathrm i\sin\alpha)$.
Transcribed to the BCK model it is  the segment connecting $\iota^{[2]}_{\mathrm{BCK}}(-\cos\alpha\pm \mathrm i\sin\alpha)=(-\cos\alpha, 0)$
 and $\iota^{[2]}_{\mathrm{BCK }}(\cos\alpha+\mathrm i\sin\alpha)=(\cos\alpha, 0)$; also yielding the statement.

Similarly, one finds
\[
\bem
\Rea\langle S_{\beta}\mathbf x,\mathbf x\rangle \\ 0 \\ \langle S_{\beta}\mathbf x,S_{\beta}\mathbf x\rangle \\ \langle \mathbf x,\mathbf x\rangle
\eem
=
\bem
\frac12\cos\beta&0&0&0\\
0&0&0&0\\
0&0&\frac12&\frac12\\
0&0&0&1
\eem
\bem
2\Rea( z_2\bar z_1)\\
2\Ima( z_2\bar z_1)\\
|z_2|^2-|z_1|^2\\
|z_2|^2+|z_1|^2
\eem.
\]
If $\cos\beta>0$, then this yields the equation
\[
\frac{(x_{\mathrm{pCK}})^{2}}{\left(\frac12\cos\beta\right)^2}+
\frac{\left(z_{\mathrm{pCK}} -\frac12\right)^{2}}{\left(\frac12\right)^2} \leq1.
\]
The transcription
\begin{equation}
\frac{(x_{\mathrm{BCK }})^{2}}{\left(\frac{\sqrt2}2\cos\beta\right)^2}+
\frac{\left(z_{\mathrm{BCK }} +\frac12\right)^{2}}{\left(\frac12\right)^2} \leq1,
\plabel{eq:CRS}
\end{equation}
to the BCK model is also straightforward, yielding the statement.
For $\cos\beta=0$, i.~e.~$\beta=\pi/2$, $\DW_{\mathrm{pCK}}^{\mathbb R}(L_{\alpha,\pi/2}^\pm)$
 yields the segment connecting $(0,0)=\iota^{[2]}_{\mathrm{pCK}}(0)$ and $(0,1)=\iota^{[2]}_{\mathrm{pCK}}(1)$.
Transcribed to the BCK model, this yields the segment connecting
$\iota^{[2]}_{\mathrm{BCK }}(0)=(0,-1)$ and $\iota^{[2]}_{\mathrm{BCK }}(1)=(0,0)$; also yielding the statement.

The case of $\mathbf 0_2$ is trivial.

(Remark: The $\DW^{\mathbb R}_{\mathrm{pCK}}(A)$'s in question could also have been computed from
 \eqref{eq:doubac}, but, ultimately, I do not find that simpler.)
\end{proof}
\snewpage
\begin{remark}
\plabel{rem:CRall}
Using Theorem \ref{thm:CRtrans}, one can show that the conformal range in the BCK model
 yields all possible degenerate elliptical disks in the unit disk which do not contain the ``infinity'' point
  $\iota^{[2]}_{*}(\infty)=(0,1)$.
The statement is essentially that, using an $h$-collineation, any the elliptical
 disk strictly contained in the unit disk  can be brought into standard form considered above.
We sketch the main points of the argument here.
The nontrivial  cases are when the elliptical disks are non-degenerate:

(a)
Assume that the elliptical disk with equation $Q(x_{\mathrm{BCK}},z_{\mathrm{BCK }})\leq0$
 is contained interior of the unit disk, the latter is of equation $Q_0(x_{\mathrm{BCK}},z_{\mathrm{BCK }})
 =1-(x_{\mathrm{BCK}})^2-(z_{\mathrm{BCK}})^2\geq0$.
Then, for an appropriate $r>0$, the
 quadric $Q'=Q+rQ_0$ will be a point ellipse, with supporting point in the interior of the unit disk.
Then, for an appropriate Möbius transformation $f$,
 this point can be centered in the unit disk to $(0,0)$; i. e.
 $f^{\star}_{\mathrm{BCK}}Q'$ will be centered at the unit disk, while
 $f^{\star}_{\mathrm{BCK}}Q_0=Q_0$.
(In fact, $f^{\star}_{\mathrm{BCK}}$ is induced a by linear transformation $R_{\mathrm{BCK}}^{[2]}(f)$
 leaving the projectivized quadratic form $(w_{\mathrm{BCK}})^2 -(x_{\mathrm{BCK}})^2-(z_{\mathrm{BCK}})^2$
 invariant, etc.)
Then  $f^{\star}_{\mathrm{BCK}}Q=f^{\star}_{\mathrm{BCK}}Q'-rQ_0$ will also be centered at $(0,0)$.
Applying a further rotation around the origin, even the axes of $g^{\star}_{\mathrm{BCK}}f^{\star}_{\mathrm{BCK}}Q$ will be aligned
 to the $x,z$-axes.
 According to the previous theorem such well-aligned ellipses will occur as conformal ranges associated to $L^\pm_{\alpha,t}$
 ($\alpha\in(0,\pi/2]$).
(Here $f^{\star}_{\mathrm{BCK}}Q'$ was a point circle or a proper point ellipse, that
distinguishes the case of $h$-circles and proper $h$-ellipses ultimately.)

(b)
Assume that the elliptical disk with equation $Q(x_{\mathrm{BCK}},z_{\mathrm{BCK }})\leq0$
in the unit disk, but with one asymptotic (boundary point) $B$.
Then, for an appropriate $r>0$, we have the
 quadric $Q'=Q+rQ_0$  so that the gradient of $Q'$ will vanish at $B$; yet it will be supported
 on the boundary only at $B$.
Therefore $Q'$ will either be a double line tangent to the unit disk at $B$, or a point ellipse at $B$.
In the first case, using a $h$-collineation, $B$ can be brought to $(0,-1)$.
Then $f^{\star}_{\mathrm{BCK}}Q'$ will be a positive multiple of $((z_{\mathrm{BCK}}+1)^2)$.
(This is the horocyclic case.)
In any way, $f^{\star}_{\mathrm{BCK}}Q=f^{\star}_{\mathrm{BCK}}Q'-rQ$
 will have an interior point along line $x_{\mathrm{BCK}}=0$.
Using an $h$-translation along this line, this point can be brought to the origin $(0,0)$.
Then  $g^{\star}_{\mathrm{BCK}}f^{\star}_{\mathrm{BCK}}Q$ vanishes at $(0,0)$, so it will be a positive multiple of
 $(z_{\mathrm{BCK}}+1)^2-(1-(x_{\mathrm{BCK}})^2-(z_{\mathrm{BCK}})^2)=(x_{\mathrm{BCK}})^2+2(z_{\mathrm{BCK}})^2+2z_{\mathrm{BCK}}$
 served by $S_0$.
In the second case, $Q'=Q_t+Q_c$,
 where   $Q_t$ and $Q_c$ quadratic forms corresponding to double lines (i.~e.~they are of rank $1$),
 and $Q_t$ is chosen to be  tangent line at $B$. (This prescription determines the decomposition.)
The (double) line corresponding to $Q_c$ is the tangent-conjugate axis of $Q'$.
Again, using a collineation, $B$ can be brought to $(0,-1)$, and the tangent-conjugate axis can be brought to
 $x_{\mathrm{BCK}}=0$.
Then $f^{\star}_{\mathrm{BCK}}Q'$ is a positive linear combination of $(x_{\mathrm{BCK}})^2$
 and $(z_{\mathrm{BCK}}+1)^2$; and $f^{\star}_{\mathrm{BCK}}Q$ also contains $Q_0$ in the linear combination.
Applying a hyperbolic translation along $x_{\mathrm{BCK}}=0$, we can assume that $(0,0)$
 is on $g^{\star}_{\mathrm{BCK}}f^{\star}_{\mathrm{BCK}}Q=0$.
Then $g^{\star}_{\mathrm{BCK}}f^{\star}_{\mathrm{BCK}}Q$ will be a positive linear combination of  $(x_{\mathrm{BCK}})^2$
 and $(z_{\mathrm{BCK}}+1)^2-(1-(x_{\mathrm{BCK}})^2-(z_{\mathrm{BCK}})^2)=(x_{\mathrm{BCK}})^2+2(z_{\mathrm{BCK}})^2+2z_{\mathrm{BCK}}$.
 Again, short inspection yields that the $S_\beta$ ($\beta\in(0,\pi/2)$) realize these.
 (Here $Q'$ at $B$ was of rank $1$ or $2$, this distinguished the $h$-horocyclic and $h$-elliptic parabolic cases.)

(c)
Assume that the elliptical disk with equation $Q(x_{\mathrm{BCK}},z_{\mathrm{BCK }})\leq0$
in the unit disk, but with asymptotic (boundary) points $B_1$ having $B_2$.
Then $Q'=Q+rQ_0$ can be achieved to be a double line passing trough $B_1$, $B_2$, etc.,
and we have the $h$-distance band case corresponding to $L_t=L^\pm_{\alpha,t}$.
 with $\alpha=0$.

Ultimately, the conformal ranges of $2\times2$ complex matrices can be identified as possibly degenerate $h$-ellipses
 and $h$-elliptic parabolas in the asymptotically closed plane
 but which avoid the distinguished asymptotical point $\iota^{[2]}_{*}(\infty)$.
\qedremark
\end{remark}

\snewpage
\begin{disc}
\plabel{disc:CRall}

The conformal ranges of the canonical representatives are particularly suitable to clarify some geometric data;
 namely symmetries and related metric data.
(This is, of course, under the paradigm of \textit{not} having preliminary knowledge about $h$-conics.)
Firstly, we establish some terminology regarding major and minor axes.
Unfortunately these latter terms might refer to possibly infinite segments (primarily), but also to lines (of these segments), or to lengths (of these segments).
Hopefully, this will not cause too much confusion.

Consider $\DW^{\mathbb R}_{\mathrm{BCK}}(L_{\alpha,t}^\pm)$:
\begin{itemize}
\item
If $\alpha=0$, then the range is a $h$-distance band, which is recognizable from its position and shape.
(Regarding its boundary, it is like \eqref{eq:chorocc} but rotated around the origin by $90^\circ$, now
with asymptotic points $(\pm1,0)$.)
The major axis [segment] can be defined as the single $h$-line contained in $\DW^{\mathbb R}_{\mathrm{BCK}}(L_{\alpha,t}^\pm)$.
It allows several minor axes [segments] which are intersections of the range with any line perpendicular to the major axis.
Note, however, that minor axes [lines] are defined naturally, even if the minor axes [segments] degenerate to points.

\item
If  $0<\alpha<\pi/2$, then the range is bounded but not circular.
It is contained in a (unique) disk of minimal radius in Euclidean view.
Then it is easy to see that it is also the (unique) $h$-disk of minimal radius containing it,
as a segment in range is a diameter of the $h$-disk (thus a smaller $h$-disk containing it is impossible).
This diameter is the major axis [segment].
The intersection of the range and the  perpendicular bisector of the major axis [segment] is the minor axis [segment].
Note that the line of minor axis well-defined, even if the minor axis [segment] reduces to a point.
As it is different from the minimal $h$-disk containing it, we must call this a proper $h$-elliptical disk.

\item
If  $\alpha=\pi/2$, then the range is bounded and circular, it is an $h$-disk.
As this is a familiar case, we simply note that there are several possible major and minor axes [lines] perpendicular to each other.
\end{itemize}

Consider $\DW^{\mathbb R}_{\mathrm{BCK}}(S_\beta)$:
\begin{itemize}
\item
If $\beta=0$, then we have a horodisk.
(One can see that the equation of its boundary,
\[\left(x_{\mathrm{BCK}}\right)^2+2\left(z_{\mathrm{BCK}}\right)^2+2\,z_{\mathrm{BCK}}=0\]
is just \eqref{eq:choro}, but rotated around the origin by $180^\circ$, now with asymptotic point $(0,-1)$.)
Major axes [segments] are intersections of the horodisk with its axes.
Minor axes [segments or lines] are not defined at all, but the ``length of the minor axis'' is set to be $+\infty$.

\item
If  $0<\beta<\pi/2$, then the range is a subset of the previous case.
It has the has the same asymptotic point, but in contrast to it, it has a single
point ``farthest'' from the asymptotic point (as all other boundary points are contained in the interior of the previous horodisk).
Now any symmetry axis must go through the asymptotic point, and also through the single point farthest
from the asymptotic point. Thus the range must have a single symmetry axis (obviously given by $x_{\mathrm{BCK}}=0$).

The major axis [segment] is its intersection with its single symmetry axis.
Minor axes [segments or lines] are not defined at all, but the ``length of the minor axis'' is set as follows.
We consider the intersections of the range with the lines perpendicular to the major axis.
Intersections will yield segments whose lengths, in the present case, limit as
\begin{multline*}
\lim_{z\searrow-1} \mathrm d^{\mathrm{BCK}}\left( \left( \sqrt{-2z(1+z)}\cos\beta,z\right), \left( -\sqrt{-2z(1+z)}\cos\beta,z\right)   \right)=\\
=\lim_{z\searrow-1}\arcosh \frac{1-z-2z(\cos\beta)^2}{1-z+2z(\cos\beta)^2}=\arcosh \frac{1+(\cos\beta)^2}{1-(\cos\beta)^2}=2\arcosh \frac1{\sin\beta}.
\end{multline*}
This limit, or even better, supremum,  can be defined as the (properly: asymptotic) length of the minor axis.
In fact, by the previous computation we even know that the range part of a distance band around the axis.
Considering the monotonicity relations in the limit, we see that the minimal radius
 for such a distance band is $\arcosh \frac1{\sin\beta}$, the semi-axis length.
An alternative argument to the same end is as follows.
We see that $\DW^{\mathbb R}_{\mathrm{BCK}}(S_\beta)$ is obtained from $\DW^{\mathbb R}_{\mathrm{BCK}}(S_0)$ by taking
 a perpendicular affinity of ratio $\cos\beta$ to the axis $x_{\mathrm{BCK}}=0$.
Therefore, $\DW^{\mathbb R}_{\mathrm{BCK}}(S_\beta)$ in contained in image of the unit disk
 through the perpendicular affinity of ratio $\cos\beta$ to the axis $x_{\mathrm{BCK}}=0$.
This is an elliptical disk with axes $\{0\}\times[-1,1]$ and $[-\cos\beta,\cos\beta]\times\{0\}$.
This is $h$-distance band.
One can compute relatively simply that its radius is $\arcosh\frac1{\sin\beta}$.
One can also see that $\DW^{\mathbb R}_{\mathrm{BCK}}(S_\beta)$ cannot be contained in any more meager $h$-distance band,
 because that would contradict to the (hyper)osculating nature of $\DW^{\mathbb R}_{\mathrm{BCK}}(S_0)$.
Also for this reason, every more meager $h$-distance band is contained in the
range around the relevant asymptotic point.

We  define the vertex of the range as the
 non-asymptotical endpoint of the major axis [segment], which is also the farthest point from the asymptotic point.

In any case, the range has an asymptotic point, it is not a horodisk, thus it is  an $h$-elliptical parabolic disk.

\item
If  $0<\beta=\pi/2$, then the range is a half-line, the major axis length is also infinite, and the ``length of the minor axis'' is $0$.
\end{itemize}
Consider $\DW_*(\mathbf 0_2)$:
\begin{itemize}
\item
The major axis [segment] is the range itself.
The major axis [line] can be considered as any line asymptotical to the range.
Minor axes [segments or lines] are not defined at all.
The lengths of the major and minor axes can be considered as $0$.
\end{itemize}
Although we have just considered the canonical examples, the definitions above were set up in terms of hyperbolic geometry
(using symmetries and distances).
Consequently, they make sense generally (i.~e.~up to $h$-congruences), without any reference to any theory of $h$-conics.
In particular, in all cases we have defined the lengths of the major and minor axes, and
 we have identified  natural symmetries.
\end{disc}
\begin{proof}[Proof of Theorem \ref{thm:CRdonc}]
Due to the natural transformation properties (as in Theorem \ref{thm:CRtrans},
 Lemma \ref{lem:eigdistinv}) it is sufficient to prove this
 only for the canonical representatives of Lemma \ref{lem:preCR}.
The previous Discussion \ref{disc:CRall} (among others things) identifies the geometric type of the conformal range
 (like asymptotically closed $h$-half line, or $h$-elliptic parabolic disk, etc.).
Comparing Lemma \ref{lem:eigpoints} and Theorem \ref{cor:CR} also informs us about the
relative position of the $h$-eigenpoints.
\end{proof}
\snewpage
 Next we will consider the (semi)axis lengths defined in Discussion \ref{disc:CRall}.
\begin{disc}\plabel{ex:CRcont}
In the case of $L_{\alpha,t}^\pm$, we may say that its characteristic BCK values are
\begin{commentx}
\[s^+=\artanh \frac{\sqrt{(\cos\alpha)^2+t^2}}{\sqrt{1+t^2} } ,\qquad
s^-=\artanh \frac{t}{\sqrt{1+t^2}}  ,\qquad
s^{\mathrm f}=\artanh \cos\alpha  .\]
\end{commentx}
\[\chi^+=\frac{\sqrt{(\cos\alpha)^2+t^2 }}{\sqrt{1+t^2}}  ,\qquad \chi^-=\frac{t}{\sqrt{1+t^2}}
 ,\qquad  \chi^{\mathrm e}=\cos\alpha.\]
These correspond to the major and minor semiaxes of the (possibly degenerate)
 ellipses $\DW^{\mathbb R}_{\mathrm{BCK}}(L_{\alpha,t}^\pm)$
  measured from $(0,0)$ in Euclidean view; and to the distance of $h$-eigenpoints $(\pm\cos\alpha)$ from $(0,0)$ in Euclidean view.
In terms of hyperbolic geometry, these values (the major semi-axis length, the major semi-axis length, and the half $h$-eigendistance) are
\[s^+=\arcosh \frac{\sqrt{1+t^2}}{\sin\alpha}  \qquad  s^-=\arcosh \sqrt{1+t^2}  \qquad
s^{\mathrm e}=\arcosh \frac1{\sin\alpha},\]
respectively, cf.~\eqref{eq:distex2}. (The major semi-axis is $+\infty$ for $\alpha=\pi/2$.)
If the minor semi-axis is $0$, then it may yield an $h$-point, an $h$-segment, or an $h$-line.
If the minor axis is non-zero, then it yields a $h$-distance band or an   $h$-circle or a proper $h$-ellipse.

In the case of $S_{\beta}$, we may say (by declaration) that its characteristic BCK values are
\[\chi^+=1  ,\qquad  \chi^-=\cos\beta ,\qquad   \chi^{\mathrm e}=\sgn\beta.\]
They correspond to the hyperbolic major and minor semi-axes lengths and half $h$-eigendistance
\[s^+=+\infty  ,\qquad  s^-=\arcosh \frac{1}{\sin\beta}  ,\qquad  s^{\mathrm e}=(\sgn\beta)\cdot(+\infty).\]
(The minor semi-axis is $+\infty$ for $\beta=0$.)
The minor semi-axis can be $0$, non-zero finite or $+\infty$.
Then it yields an $h$-half-line, a  $h$-elliptic parabola, or an $h$-horosphere, respectively.

The case $\mathbf 0_2$ corresponds to an asymptotical point.
The wisest choice is  to assign the characteristic values
\[\chi^+=0 ,\qquad  \chi^-=0, \qquad \chi^{\mathrm e}=0\]
and
\[s^+=0 \qquad s^-=0, \qquad s^{\mathrm e}=0.\]
In general, a  characteristical BCK value $u$ corresponds to hyperbolic semi-axis length
\[\arcosh \frac1{\sqrt{1-u^2}}=\arcosh \frac u{\sqrt{1-u^2}}=\artanh u=\frac12\arcosh\frac{1+u^2}{1-u^2},\]
cf.  \eqref{eq:distex2}. \qedexer
\end{disc}
As the representatives in Theorem \ref{cor:CR} cover all cases
up to real M\"obius transformations, i.~e.~$h$-isometries in the conformal range, the terminology
`characteristic BCK values'  based on the previous discussion can also be applied to arbitrary $2\times2$ complex matrices.
\snewpage

\begin{example}\plabel{ex:dconfrep}

Checking the canonical representatives of Lemma \ref{lem:preCR}, we find:

(a) For $A=L_{\alpha,t}^\pm$,
\[U_A-|D_A|=2t^2,\qquad \{U_A+|D_A|,U_A-|D_A|+2E_A\}=\{2((\cos\alpha)^2+t^2) , 2(1+t^2)\}.\]

(b) For $A=S_{\beta}$,
\[U_A-|D_A|=\tfrac12{(\cos\beta)^2},\qquad \{U_A+|D_A|,U_A-|D_A|+2E_A\}=\left\{\tfrac12,\tfrac12\right\}.\]

(c) For $A=\mathbf0_2$,
\[U_A-|D_A|=0,\qquad \{U_A+|D_A|,U_A-|D_A|+2E_A\}=\left\{0,0\right\}.\eqedexer\]
\end{example}

\begin{disc}\plabel{disc:confrep}
The triple ratio
\begin{equation}
U_A-|D_A|\,\,:\,\,U_A+|D_A|\,\,:\,\,U_A-|D_A|+2E_A
\plabel{eq:confrep}
\end{equation}
is just as appropriate in Corollary \ref{cor:rconftriple} as \eqref{eq:preconfrep}
(having the invertible matrix  $\left[\begin{smallmatrix}1&1&1\\-1&1&-1\\0&0&2
\end{smallmatrix}\right]$ applied to the ratio).
As the previous example shows, making indistinction in the order of the last two entries of  \eqref{eq:confrep},
 with some abuse of notation,
\[
U_A-|D_A|\,\,:\,\{\,U_A+|D_A|\,\,:\,\,U_A-|D_A|+2E_A\,\}
\]
captures the commonality of the classes of $L^+_{\alpha,t}$ and $L^-_{\alpha,t}$.
Actually, as $U_A-|D_A|$ is distinguished as the entry of the smallest absolute value,
the same can be said about the unordered ratio
\[
\{\,U_A-|D_A|\,\,:\,\,U_A+|D_A|\,\,:\,\,U_A-|D_A|+2E_A\,\}.\eqed
\]
\end{disc}

\begin{theorem}\plabel{thm:finv1}
The unordered ratio
\begin{equation}
\{U_A-|D_A| \,:\, U_A+|D_A| \,:\, U_A-|D_A|+2E_A\}
\plabel{eq:finv1}
\end{equation}
and (when $|D_A|=E_A>0$) the possible choice of type of possibly degenerate
\[\text{$h$-distance band / $h$-elliptic parabolic disk}\]
together form a full invariant of the conformal range up to $h$-isometries.

If \eqref{eq:finv1} is $\{\chi_1:\chi_2:\chi_3\}$ with $0\leq\chi_1\leq\chi_2\leq\chi_3$,
 then ``characteristic BCK values'' can be obtained as follows:
\[\chi^+=\sqrt{\frac{\chi_2}{\chi_3}},\qquad
\chi^-=\sqrt{\frac{\chi_1}{\chi_3}},\qquad
\chi^{\mathrm e}=\sqrt{\frac{\chi_2-\chi_1}{\chi_3-\chi_1}}.\]
(Here the convention $\frac00=0$ applies.)

From a ``characteristic BCK value'' $u$, the corresponding major or minor semi-axis length
or half $h$-eigendistance is
$\arcosh \frac1{\sqrt{1-u^2}}=\arsinh \frac u{\sqrt{1-u^2}}=\artanh u=\frac12\arcosh\frac{1+u^2}{1-u^2}$.
\begin{proof}
The first part of the statement is Corollary \ref{cor:rconftriple} and the previous Discussion \ref{disc:confrep} taken together.
The second part is  Discussion \ref{ex:CRcont} in view of Example \ref{ex:dconfrep}.
\end{proof}
\end{theorem}
\snewpage
The following statement expresses the information regarding the major and minor (semi-)axes
in somewhat more down-to-earth terms (similarly to Theorem \ref{lem:eigdist}).
\begin{theorem}\plabel{rem:recoverBCK}
Assume that $A$ is a $2\times2$ complex matrix.

(a) If $E_A\geq |D_A|$ (the inequality is strict in the complex-parabolic and quasihyperbolic cases), then  the characteristic BCK values are
\[ \chi^+=\sqrt{\frac{U_A+|D_A|}{U_A-|D_A|+2E_A}} \qquad \geq \qquad\chi^-= \sqrt{\frac{U_A-|D_A|}{U_A-|D_A|+2E_A}} .\]
The corresponding $h$-distances (the semi-axes), the $\artanh$'s of the values above, are
\[s^+=\frac12\arcosh \frac{U_A+E_A}{E_A-|D_A|}\qquad \geq \qquad s^-=\frac12\arcosh \frac{U_A-|D_A|+E_A}{E_A}.\]

(b) If $|D_A| \geq E_A$ (the inequality is strict in the real-elliptic and quasielliptic cases), then  the characteristic BCK values are
\[ \chi^+=\sqrt{\frac{U_A-|D_A|+2E_A}{U_A+|D_A|}} \qquad \geq \qquad\chi^-=\sqrt{\frac{U_A-|D_A|}{U_A+|D_A|}  } .\]
The corresponding $h$-distances (the semi-axes), the $\artanh$'s of the values above, are
\[s^+=\frac12\arcosh \frac{U_A+E_A}{|D_A|-E_A}\qquad \geq \qquad s^-=\frac12\arcosh \frac{U_A}{|D_A|}.\]

(c) In the case $E_A=|D_A|$ (encompassing the real-parabolic, real-hyperbolic, and semi-real cases), the characteristic BCK values are
\[ \chi^+=1 \qquad \geq \qquad\chi^-=\sqrt{\frac{U_A-|D_A|}{U_A+|D_A|}  }\equiv \sqrt{\frac{U_A-E_A}{U_A+E_A}  } .\]
The corresponding $h$-distances (the semi-axes), the $\artanh$'s of the values above, are
\[s^+=+\infty \qquad \geq \qquad s^-=\frac12\arcosh \frac{U_A}{|D_A|}\equiv\frac12\arcosh \frac{U_A}{E_A}.\]

(At the characteristic BCK values the convention $\frac00=0$ was used;
at the $h$-distances, the convention $\frac00=1$ was used.)

\begin{proof}
This as in the previous theorem but the characteristic BCK value $u$
 is transcribed to the corresponding semi-axis length as $ \frac12\arcosh\frac{1+u^2}{1-u^2}$.
\end{proof}

\end{theorem}
\begin{remark}
\plabel{rem:crossing}
Note that in case (b),  if  $|D_A|>E_A$, then the axis of the Davis--Wielandt shell crosses the real hyperbolic plane,
 the minor axis $s^-$ will be the same as the radius of the Davis--Wielandt shell
 (there will be an actual radial segment laying on the real hyperbolic plane).
In case (c), i.~e.~if $|D_A|=E_A$, the real hyperbolic plane still have, at least asymptotically,
  an axial point of the Davis--Wielandt shell (there are several axes in the parabolic case), the same is true about $s^-$.
In case (a), if  $E_A>|D_A|$, then the Davis--Wielandt shell stays well away from the real hyperbolic plane.
Then the projection to the real hyperbolic plane is dominantly contractive with respect to radial segments 
(we do not define this precisely),
 and $s^-$  will indeed be smaller than the radius of  the Davis--Wielandt shell.
The values $s^+$ and $s^{\mathrm e}$ behave formally quite similarly.
      \qedremark
\end{remark}


\snewpage
\subsection{The synthetic geometry of the conformal range}
\plabel{ssub:CRsynth}
~\\

Here we will consider the synthetic interpretations of our (originally analytically given) $h$-ellipses and
$h$-elliptic parabolas.

(X) Assume that $A$ is a complex $2\times2$ matrix with no real eigenvalues.
Let $ \Lambda_1 $ and $ \Lambda_2 $ be its $h$-eigenpoints.
Let $\boldsymbol x $ be a hyperbolic point (with coordinates $\boldsymbol x_{\mathrm{BCK }}=( x_{\mathrm{BCK }} , z_{\mathrm{BCK }})$
in the $\mathrm{BCK }$ model).
We will use the abbreviations
\[f_1=\mathrm d ( \boldsymbol x ,  \Lambda_1  ),\qquad
f_2=\mathrm d (  \boldsymbol x,  \Lambda_2  ),\qquad
m^+=\arcosh \frac{U_A+E_A}{\left| |D_A|-E_A\right|},\]
the latter value being the length of the major axis of $\DW^{\mathbb R}(A)$
computed according to Theorem \ref{rem:recoverBCK}.

\begin{theorem}
\plabel{cor:CRER1}
Let us consider the setup of (X).
Then the  points $\boldsymbol x $ of $\DW ^{\mathbb R}(A)$
are described as the solutions of the inequality
\begin{equation}
f_1+f_2\leq m^+,
\plabel{eq:ellipo}
\end{equation}
with equality for the boundary.
Thus, the boundary is the synthetic $h$-ellipse with the $h$-eigenpoints as the foci and the length of the major axis as the distance sum.
\end{theorem}

\begin{proof}[Proof without previous knowledge on $h$-ellipses.]
It is sufficient to check to the canonical representatives.
This means $L_{\alpha,t}^\pm$ with $0< \alpha\leq\pi/2$.
Then
\[(\Lambda_1)_{\mathrm{BCK}}=(\cos\alpha,0),
\qquad\text{and}\qquad
(\Lambda_2)_{\mathrm{BCK}}=(-\cos\alpha,0). \]
Furthermore,
\[\frac{U_A+E_A}{\left| |D_A|-E_A\right|}=\frac{1+2t^2+(\cos\alpha)^2}{(\sin\alpha)^2}.\]
The normal case ($t=0$) is trivial, so we will assume that the elliptical disk is non-degenerate (thus $t>0$).
~
Now, \eqref{eq:ellipo} is equivalent to
\begin{equation}
\cosh(f_1+f_2)-\cosh m^+\leq0.
\plabel{eq:ellipo2}
\end{equation}
As $|f_1-f_2|<m+$, we know $\cosh m^+-\cosh(f_1+f_2)>0$,
thus the inequality above is fully equivalent to
\begin{equation}
\left( \cosh(f_1+f_2)-\cosh m^+ \right)\left( \cosh m^+ -\cosh(f_1-f_2) \right)\leq0.
\plabel{eq:ellipo3}
\end{equation}
Using standard addition formulae, this transcribes as
\begin{equation}
 1- (\cosh f_1)^2-(\cosh f_2)^2-(\cosh m^+)^2+2(\cosh f_1)(\cosh f_2)(\cosh m^+)\leq0.
 \plabel{eq:ellipo4}
\end{equation}
(The LHS is a symmetric expression in $f_1,f_2,m^+$.)
This form is particularly suitable to substitute the concrete expressions of
$f_1=\mathrm d^{\mathrm{BCK}}(( x_{\mathrm{BCK }} , z_{\mathrm{BCK }}), (\cos\alpha,0) )$,
$f_2=\mathrm d^{\mathrm{BCK}}(( x_{\mathrm{BCK }} , z_{\mathrm{BCK }}), (-\cos\alpha,0) )$,
$m^+=\arcosh\frac{1+2t^2+(\cos\alpha)^2}{(\sin\alpha)^2}$ in the BCK model, which are $\arcosh$'s.
Then \eqref{eq:ellipo4} transcribes to
\[\frac4{(\sin\alpha)^4}\cdot
\frac{ t^2(1+t^2)(x_{\mathrm{BCK}})^2+ ((\cos\alpha)^2+t^2)(1+t^2)(z_{\mathrm{BCK}})^2 - t^2((\cos\alpha)^2+t^2) }{
1- (x_{\mathrm{BCK}})^2-(z_{\mathrm{BCK}})^2 }\leq0.
\]
This is equivalent to $\eqref{eq:CRL}$ (for ordinary hyperbolic points).
\end{proof}
This, in particular, proves that our (analytic) $h$-ellipses are synthetic $h$-ellipses.
In fact, in light of Remark \ref{rem:CRall}, it proves that all (analytic) $h$-ellipses are synthetic $h$-ellipses.

If we are a bit more familiar with $h$-conics, but we still pursue the synthetic approach to the foci,
 then the statement above can be obtained in another way:

\begin{proof}[Proof with  previous knowledge on $h$-ellipses.]
We know that we deal with an $h$-elliptical disk ($h$-circle allowed).

Let $s^+$ be the $h$-length of its major semi-axis, $s^-$  be the $h$-length of its minor semi-axis,
and let $s^{\mathrm e}$ be half of the $h$-eigendistance of $A$
We claim that
\begin{equation}
\cosh s^+=(\cosh s^-)(\cosh s^{\mathrm e})
\plabel{eq:neus}
\end{equation}
holds.
It is sufficient to check the identity squared.
That is
\begin{equation}
\frac{1+\cosh 2s^+}2=\frac{1+\cosh 2s^-}2\frac{1+\cosh 2s^{\mathrm e}}2.
\plabel{eq:sassa}
\end{equation}
This, however, is already transparent from Theorem \ref{rem:recoverBCK} and Theorem \ref{lem:eigdist}.
Indeed,
if $E_A> |D_A|$ then \eqref{eq:sassa} is equivalent to
\[\frac{1+\frac{U_A+E_A}{E_A-|D_A|}}2=\frac{1+\frac{U_A-|D_A|+E_A}{E_A}}2\cdot \frac{1+\frac{ E_A+|D_A|}{E_A-|D_A|}}2 ;\]
if $|D_A| > E_A$   then \eqref{eq:sassa} is equivalent to
\[\frac{1+ \frac{U_A+E_A}{|D_A|-E_A}}2=\frac{1+ \frac{U_A}{|D_A|}}2\cdot \frac{1+\frac{ E_A+|D_A|}{|D_A|-E_A}}2 .\]
Each case is easy to check.

Now, comparing \eqref{eq:neus} to \eqref{eq:Pellip}
 implies that $ s^{\mathrm e}=s^{\mathrm f}$, thus the focal distance is equal to the $h$-eigendistance.
As we know from the BCK picture that both the $h$-eigenvalues and the foci must be in symmetric positions
on the major axis, and they are of the same distance from each other, they must be the same.
\end{proof}

(Y)
Suppose that $A$  is a complex $2\times 2$ matrix with a strictly complex eigenvalue $\lambda$ and a real eigenvalue $\lambda_0$.
Let the corresponding $h$-eigenpoints be $\Lambda$ and $\Lambda_0$, respectively.
Using the canonical representatives $S_\beta$ with $\beta\in(0,\pi/2]$, we
 know $\DW^{\mathbb R}(A)$ is an $h$-elliptic parabolic disk has single symmetry axis, containing
$\Lambda$ and $\Lambda_0$ and a vertex $V$.
Our objective is to demonstrate that this $h$-elliptic parabolic disk has an appropriate synthestic presentation with focus $\Lambda$.
Let $\mathrm d_0$ be a distance function from $\Lambda_0$.
This is unique only up to  adding constant scalar.

Let $\boldsymbol x$ denote hyperbolic points in general.
Then the function
\[\boldsymbol x\mapsto\mathrm d_0(\boldsymbol x)+\mathrm d_0(\Lambda)-2\mathrm d_0(V)\]
is also a distance function from $\Lambda_0$.
It is also the signed distance function from the horocycle
\[C=\{\tilde{\boldsymbol x}\,:\, \mathrm d_0(\tilde{\boldsymbol x})+\mathrm d_0(\Lambda)-2\mathrm d_0(V)=0\}.\]
This horocycle $C$ passes through the point $P$ which is the point $\Lambda$ reflected through $V$.
($V$ and $\Lambda$ are interior points where $\mathrm d_C$ is negative.)

\snewpage
\begin{theorem}
\plabel{cor:CRER2}
Let us consider the setup (Y).
Then the points $\boldsymbol x$ of the conformal range are given by the inequality
\begin{equation}
\mathrm d(\boldsymbol x,\Lambda) +\underbrace{\mathrm  d_0(\boldsymbol x)+\mathrm d_0(\Lambda)-2\mathrm d_0(V)
}_{\mathrm d_C(\boldsymbol x)}\leq0,
\plabel{eq:onsa}
\end{equation}
with equality on the boundary.

If $A$ is not normal, then the conformal range is also described by
\begin{equation}
 \mathrm d(\boldsymbol x,\Lambda) \leq|\mathrm d_C( \boldsymbol x)|,
 \plabel{eq:monsa}
\end{equation}
with equality on the boundary.
This boundary of the conformal range is the $h$-elliptic parabola, whose points
are of equal distance from $\Lambda$ and from the horocycle $C$ with equation
$\mathrm d_0(\tilde{\boldsymbol x})+\mathrm d_0(\Lambda)-2\mathrm d_0(V)=0$.
In other terms, this is the synthetic $h$-elliptic parabola with vertex $V$ and focus $\Lambda$.
\end{theorem}

\begin{proof}[Proof without preliminary knowledge on $h$-elliptic parabolas.]

This is sufficient to check for the canonical representatives $S_\beta$ with $\beta\in(0,\pi/2]$.
Then
\[V_{\mathrm{BCK}}=(0,0),\qquad
\Lambda_{\mathrm{BCK}}
=\left(0,\frac{(\cos\beta)^2}{(\cos\beta)^2-2}\right),
\qquad\text{and}\qquad(\Lambda_0)_{\mathrm{BCK}}=(0,-1) .\]
Moreover, transcribed from $\mathrm d_0^{\mathrm{Ph}}(x_{\mathrm{Ph}}, z_{\mathrm{Ph}})=\log{\frac{(z_{\mathrm{Ph}})^2+(z_{\mathrm{Ph}})^2}{z_{\mathrm{Ph}}}}$,
\[\mathrm d_0^{\mathrm{BCK}}(x_{\mathrm{BCK}}, z_{\mathrm{BCK}})=
\log\left(\frac{1+z_{\mathrm{BCK}} }{\sqrt{1-(x_{\mathrm{BCK}})^2-(z_{\mathrm{BCK}})^2}}\right)\]
can be used. (Then $\mathrm d_0^{\mathrm{BCK}}(V_{\mathrm{BCK}})=0$.)
Here the $h$-horocycle $C$ is the one through
$P_{\mathrm{BCK}}
=\left(0,\frac{(\cos\beta)^2}{2-(\cos\beta)^2}\right)
$
and $(0,-1)$.

First, we consider the inequality \eqref{eq:monsa} which is implied by \eqref{eq:onsa}.
Now, \eqref{eq:monsa} is fully equivalent to
\[
\cosh\bigl(\mathrm d(\boldsymbol x,\Lambda)\bigr)-
\cosh\bigl(\mathrm d_0(\boldsymbol x)+\mathrm d_0(\Lambda)-2\mathrm d_0(V)\bigr)\leq0.
\]
This, in turn, is fully equivalent to
\[\Bigl(\cosh\bigl(\mathrm d(\boldsymbol x,\Lambda)\bigr)-
\cosh\bigl(\mathrm d_0(\boldsymbol x)+\mathrm d_0(\Lambda)-2\mathrm d_0(V) \bigr)\Bigr)
\cdot{\exp(\mathrm d_0(\boldsymbol x)  - \mathrm d_0(\Lambda))}\leq0.\]

Using our special setup for $S_\beta$, however, this transcribes as
 \begin{equation}
\frac1{2(\sin\beta)^2}\cdot
\frac{(x_{\mathrm{BCK}})^2+2(\cos\beta)^2(z_{\mathrm{BCK}})^2 +2(\cos\beta)^2 z_{\mathrm{BCK}}  }{
  1-(x_{\mathrm{BCK}})^2-(z_{\mathrm{BCK}})^2 }\leq0.
   \plabel{eq:fonsa}
\end{equation}

If $\beta\in(0,\pi/2)$, then \eqref{eq:monsa}--\eqref{eq:fonsa}
 is fully equivalent to \eqref{eq:CRS}, regarding ordinary $h$-points.
 Now, \eqref{eq:CRS} implies $\mathrm d_C(\boldsymbol x)<0$.
Indeed the solution set is a subset of $\DW^{[2]}_{\mathrm{BCK}}(S_0)$, which, in turn,
 is contained in the interior of the horodisk of $C$.
Therefore \eqref{eq:CRS} fully implies (that is with equality going to equality)
$\mathrm d(\boldsymbol x,\Lambda) \leq -\mathrm d_C( \boldsymbol x)$, i.~e.~\eqref{eq:onsa}.
This shows that $\eqref{eq:CRS}$ and $\eqref{eq:onsa}$ and  $\eqref{eq:monsa}$ are fully equivalent.

If $\beta=\pi/2 $, then \eqref{eq:monsa}--\eqref{eq:fonsa} yields the line $x_{\mathrm{BCK}}=0$.
But then it is easy to check that how the more restrictive $\mathrm d(\boldsymbol x,\Lambda) +\mathrm d_C( \boldsymbol x)\leq0$
   only halves this line after that.
\end{proof}

This, in particular, proves that our (analytic) $h$-elliptical parabolas are synthetic $h$-elliptical parabolas.
In fact, in light of Remark \ref{rem:CRall}, it proves that all (analytic) $h$-elliptical are synthetic $h$-elliptical parabolas.

Assuming more familiarity with $h$-conics, the statement above can be obtained in another way:
\begin{proof}[Proof with  preliminary knowledge on $h$-elliptic parabolas.]
We start the proof as before, reducing to $A=S_\beta$.
Suppose that $s^-$ is the length of the minor semi-axis (or rather half of the length of the minor axis).
Then one can easily check that
\begin{equation}
\exp\mathrm  d(V,\Lambda)\equiv\exp\left(\mathrm  d_0(V)-\mathrm d_0(\Lambda)\right) = \cosh s^-.
\plabel{eq:CRER2}
\end{equation}
We already know that the conformal range must be a (possibly degenerate) synthetic $h$-elliptic parabola.
We know that $h$-elliptic parabolas have the property \eqref{eq:CRER2} with the focus in the place of $\Lambda$.
We also know that the tip (or vertex) is $V=(0,0)_{\mathrm{BCK}}$, and $\Lambda$ is toward $\Lambda_0$ from it.
This makes $\Lambda$ equal to the focus.
\end{proof}

Now we have

\begin{proof}[Proof of Theorems \ref{thm:addCR1} and \ref{note:addCR1}.]
Theorems \ref{cor:CRER1} and \ref{cor:CRER2} establish the synthetic interpretations in cases (i)--(ii)/(vi)--(vii) and (iv)/(ix)
 of Theorem \ref{thm:CRdonc}.
More generally, considering the asymptotic points as foci is mainly a matter of convention but one which fits well
 to various degenerating and limiting cases.
The metrical data are already contained in Theorem \ref{rem:recoverBCK} (in less compact form) and in Theorem \ref{lem:eigdist}.
\end{proof}

\begin{proof}[Proof for Theorem \ref{thm:addCR2}]
The statement is easier to see in the pCK model.
Having the foci fixed, it easy to see that we have smaller or larger geometric objects associated to them; one is contained in the other.
It is also clear that the supremum in $z_{\mathrm{pCK}}$ is the square of the norm,
 which strictly grows as the associated geometric objects become fatter.
(It ranges from the minimal $\max(|\lambda_1|^2,|\lambda_2|^2)$ toward $+\infty$.)
The transcription to the BCK model is straightforward.
\end{proof}
\begin{remark}
For $h$-circles, $h$-horospheres, $h$-distance bands, proper $h$-ellipses, $h$-elliptic parabolas
(i.~e.~in the non-normal cases), the
 $h$-foci from the quadric (of the boundary) can be obtained in the projective models as follows:
Assume that $Q$ is the quadric, let $G$ be its dual quadric (on the dual projective plane);
let $Q_0$ be the quadric of the absolute, let $G_0$ be its the dual quadric.
Then the pencil generated by $G$ and $G_0$ contains a sole quadric $G_{\mathrm{spec}}$ with is for a pair of planes
(the lines may be different from each other or not).
The lines on the dual projective plane correspond to points in original plane; these will be the $h$-foci.

Indeed, we can check this in the BCK model.
The quadric of the absolute written in projective coordinates  $x_{\mathrm{BCK }},z_{\mathrm{BCK }},w_{\mathrm{BCK }}$ is
\[Q_0:\qquad {(x_{\mathrm{BCK }})^{2}}+{(z_{\mathrm{BCK }})^{2}} -{(w_{\mathrm{BCK }})^{2}} =0.\]
The dual quadric in the dual projective coordinates  $x'_{\mathrm{BCK }},z'_{\mathrm{BCK }},w'_{\mathrm{BCK }}$ is
\begin{equation*}
G_0:\qquad  {(x'_{\mathrm{BCK }})^{2}}+  {(z'_{\mathrm{BCK }})^{2}} - {(w'_{\mathrm{BCK }})^{2}}=0.
\end{equation*}

In the case of $\partial\DW^{\mathbb R}_{\mathrm{BCK}}(L_{\alpha,t})$ ($t>0$), the quadric is
\[Q:\qquad
{\left(\dfrac {\sqrt{1+t^2}}{\sqrt{(\cos\alpha)^2+t^2}}\right)^2}  {(x_{\mathrm{BCK }})^{2}}+
{\left(\dfrac{\sqrt{1+t^2}}{t}\right)^2}   {(z_{\mathrm{BCK }})^{2}} -  {(w_{\mathrm{BCK }})^{2}}=0.
\]
The dual quadric is
 \begin{equation}G:\qquad
{\left(\dfrac{\sqrt{(\cos\alpha)^2+t^2}}{\sqrt{1+t^2}}\right)^2}  {(x'_{\mathrm{BCK }})^{2}}+
{\left(\dfrac{t}{\sqrt{1+t^2}}\right)^2}   {(z'_{\mathrm{BCK }})^{2}} - {(w'_{\mathrm{BCK }})^{2}}=0.
\end{equation}
Then, one can see, for the sole pair of lines in the pencil of $G$ and $G_0$, we have
\begin{equation}G_{\mathrm{spec}}:\qquad
  (\cos\alpha)^2{(x'_{\mathrm{BCK }})^{2}}  - {(w'_{\mathrm{BCK }})^{2}}=0.
\end{equation}
This belong to the lines $\pm(\cos\alpha){x'_{\mathrm{BCK }}}  + {w'_{\mathrm{BCK }}}=0$;
 which correspond to the points $(\pm\cos\alpha,0)$.

In the case of $\partial\DW^{\mathbb R}_{\mathrm{BCK}}(S_{\beta})$ ($\beta\in[0,\pi/2)$), the quadric is
\[Q:\qquad
 \frac{(x_{\mathrm{BCK }})^{2}}{\left( \cos\beta\right)^2}+
 2{\left(z_{\mathrm{BCK }}  \right)^{2}}+ 2z_{\mathrm{BCK }}w_{\mathrm{BCK }}=0.
\]
The dual quadric in the dual projective coordinates  $x'_{\mathrm{BCK }},z'_{\mathrm{BCK }},w'_{\mathrm{BCK }}$ is
 \begin{equation}G:\qquad
  {\left( \cos\beta\right)^2}{(x'_{\mathrm{BCK }})^{2}}+
  2z'_{\mathrm{BCK }} w'_{\mathrm{BCK }} -2(w'_{\mathrm{BCK }})^2  =0.
\end{equation}
Then, one can see, for the sole pair of lines in the pencil of $G$ and $G_0$, we have
\begin{align*}G_{\mathrm{spec}}:\qquad&
  (\cos\beta)^2  (z'_{\mathrm{BCK }})^2-2z'_{\mathrm{BCK }}w'_{\mathrm{BCK }} +(2-(\cos\beta)^2) (w'_{\mathrm{BCK }})^2\equiv
 \\
  \qquad
  &(z'_{\mathrm{BCK }}-w'_{\mathrm{BCK }})( (\cos\beta)^2z'_{\mathrm{BCK }}+((\cos\beta)^2-2)w'_{\mathrm{BCK }} )=0.
\end{align*}
This belongs to the lines $ -{z'_{\mathrm{BCK }}}  + {w'_{\mathrm{BCK }}}=0$ and
$ \frac{(\cos\beta)^2}{(\cos\beta)^2-2}{z'_{\mathrm{BCK }}}  + {w'_{\mathrm{BCK }}}=0$;
which correspond to the points $(0,-1)$ and  $\left(0,\frac{(\cos\beta)^2}{(\cos\beta)^2-2}\right)$.

Having checked these cases, projective invariance takes care for the rest.
\qedremark
\end{remark}

\begin{remark}
In fact, point ellipses can be considered as pairs of imaginary lines.
In that way they, in the situation above, they lead to imaginary focus points.
(This leads to potentially three pairs of foci, real or imaginary.)
This viewpoint is present in the theory of hyperbolic conics from the very beginning, see Story \cite{Sto}.

If we had adapted  the kind of analytic definition for the foci in the fashion of the remark above, from the beginning,
 then the synthetical identification of the foci of this section would have been not necessary.
\qedremark
\end{remark}

\begin{remark}
In the case of the conformal range of $A$, the dual quadric $G$ can be defined even if $A$ is normal.
Indeed, $G$ can be defined as $G_{\mathrm{spec}}$, which should be the product of equations of lines
 on the dual projective space corresponding the $h$-eigenpoint.
Ultimately, the dual quadric is the most convenient way to numerically code the conformal range of $A$;
 just like in the case of the numerical range.
This is not surprising if one is familiar with the Kippenhahnian view of the numerical range.
\qedremark
\end{remark}

\snewpage
\section{Comparison of the elliptical range theorems}
\plabel{sec:com}
The elliptic range theorem for the numerical range and for the Davis--Wielandt shell
 are quite close to each other.
Regarding their shape, there are $2\times2$ cases depending on (non)normality and on (non)parabolicity.
In the case of the numerical range, the four types are (i) point, (ii) segment, (iii) disk, (iv) proper elliptical disk.
In the case of the  Davis--Wielandt shell, the four types are
 (i) asymptotic $h$-point, (ii) asymptotically closed $h$-line, (iii) asymptotically closed horosphere, (iv) asymptotically closed $h$-tube.
These cases correspond to each other.
In contrast, the $2\times 5$ cases for the shape of the conformal range are more complicated.

In higher dimensions, the Davis--Wielandt shell contains more or equal information than
 either the numerical range or the conformal range.
In dimension $2$, the Davis--Wielandt shell and the numerical range contain the same
 amount of information, while the conformal range contains a reduced amount of information.
 Indeed, in the case of the numerical range, the normal form \eqref{eq:canon2} can be
 be recovered immediately; $\lambda_1$ and $\lambda_2$ as the foci, $t$ from the length of the minor axis.
This means that the numerical range is unique up to unitary conjugation of $2\times2$ matrices.
Obviously, the Davis-Wielandt shell also cannot contain more information than up to unitary conjugation.
Thus they are equivalent in the $2\times 2$ case.
However, there is a loss of information in the case of conformal range, as it folds the
real-elliptic / quasielliptic and complex-parabolic / quasihyperbolic cases together.
(This is why we have $2\times 5$ and not $2\times7$ many cases for the shape.)

Therefore, in the $2\times2$ case, one can argue that the conformal range is both more complicated and primitive than
 the numerical range and  the Davis--Wielandt shell.
The main question is, however, not the information contained (as the matrix $A$ itself contains even more),
 but that how the ranges can be utilized.
In that regard, for the conformal range, see \cite{L2}.
\snewpage

\end{document}